\documentclass[11pt]{article}

\usepackage[T1]{fontenc}
\usepackage[latin9]{luainputenc}
\usepackage{geometry}
\usepackage{color}
\usepackage{mathrsfs}
\usepackage{mathtools}
\usepackage{dsfont}
\usepackage{amsmath}
\usepackage{amsthm}
\usepackage{amssymb}
\usepackage{graphicx}
\usepackage{bbm}
\usepackage{tikz}
\usepackage{subfig}
\usepackage{epsfig}
\usepackage{caption}
\usetikzlibrary{decorations.pathreplacing,decorations.markings}
\usetikzlibrary{arrows}

\usepackage{pgfplots}
\tikzset{middlearrow/.style={
		decoration={markings,
			mark= at position 0.5 with {\arrow{#1}} ,
		},
		postaction={decorate}
	}
}

\tikzset{
	on each segment/.style={
		decorate,
		decoration={
			show path construction,
			moveto code={},
			lineto code={
				\path [#1]
				(\tikzinputsegmentfirst) -- (\tikzinputsegmentlast);
			},
			curveto code={
				\path [#1] (\tikzinputsegmentfirst)
				.. controls
				(\tikzinputsegmentsupporta) and (\tikzinputsegmentsupportb)
				..
				(\tikzinputsegmentlast);
			},
			closepath code={
				\path [#1]
				(\tikzinputsegmentfirst) -- (\tikzinputsegmentlast);
			},
		},
	},
	mid arrow/.style={postaction={decorate,decoration={
				markings,
				mark=at position .5 with {\arrow[#1]{stealth}}
	}}},
}

\usepackage{float}

\usepackage[unicode=true,
bookmarks=true,bookmarksnumbered=true,bookmarksopen=false,
breaklinks=false,pdfborder={0 0 1},backref=false,colorlinks=true]
{hyperref}
\hypersetup{
	linkcolor=blue, citecolor=red, urlcolor=blue, filecolor=blue, pdfpagelayout=OneColumn, pdfnewwindow=true, pdfstartview=XYZ, plainpages=false}

\makeatletter
\numberwithin{equation}{section}
\theoremstyle{plain}
\newtheorem{thm}{\protect\theoremname}[section]
\theoremstyle{remark}
\newtheorem{rem}[thm]{\protect\remarkname}
\theoremstyle{definition}
\newtheorem{defn}[thm]{\protect\definitionname}
\theoremstyle{plain}
\newtheorem{lem}[thm]{\protect\lemmaname}
\theoremstyle{plain}

\theoremstyle{remark}
\newtheorem*{rem*}{\protect\remarkname}
\newtheorem{ex}[thm]{\protect\examplename}

\theoremstyle{plain}
\newtheorem{cor}[thm]{\protect\corollaryname}

\usepackage{ifpdf} % part of the hyperref bundle
\ifpdf % if pdflatex is used

\IfFileExists{lmodern.sty}{\usepackage{lmodern}}{}

\fi % end if pdflatex is used

\usepackage[figure]{hypcap}
\usepackage{bbm}

\let\myTOC\tableofcontents
\renewcommand\tableofcontents{%
	\frontmatter
	\pdfbookmark[1]{\contentsname}{}
	\myTOC
	\mainmatter }

\usepackage{calligra}
\usepackage{fancyhdr}
\@ifundefined{extrarowheight}
{\usepackage{array}}{}
\DeclareMathAlphabet{\mathpzc}{OT1}{pzc}{m}{it}
\date{ }

\def\definitionname{Definition}
\def\theoremname{Theorem}
\def\propositionname{Proposition}
\def\lemmaname{Lemma}
\def\corollaryname{Corollary}
\def\remarkname{Remark}
\def\examplename{Example}

\providecommand{\keywords}[1]{\textbf{Key words:} #1}
\providecommand{\AMSclass}[1]{\textbf{AMS subject classification:} #1}
\makeatother

\providecommand{\definitionname}{Definition}
\providecommand{\lemmaname}{Lemma}
\providecommand{\propositionname}{Proposition}
\providecommand{\remarkname}{Remark}
\providecommand{\theoremname}{Theorem}
\newcommand{\N}{\mathbb N}
\newcommand{\R}{\mathbb R}

\newcommand{\cP}{\mathcal P}
\newcommand{\cS}{\mathcal S}
\newcommand{\cM}{\mathcal M}
\newcommand{\cT}{\mathcal T}
\newcommand{\cV}{\mathcal V}
\newcommand{\ccV}{\mathscr V}
\newcommand{\cJ}{\mathcal J}
\newcommand{\cU}{\mathcal U}
\newcommand{\cW}{\mathcal W}
\newcommand{\zV}{\mathpzc V}
\newcommand{\zv}{\mathpzc v}
\newcommand{\zc}{\mathpzc c}
\newcommand{\zC}{\mathpzc C}
\newcommand{\zU}{\mathpzc U}
\newcommand{\zu}{\mathpzc u}
\newcommand{\zW}{\mathpzc W}
\newcommand{\zw}{\mathpzc w}
\newcommand{\tu}{\mathfrak u}
\newcommand{\supp}{\operatorname{supp}} 
\def\ds{\displaystyle}
\definecolor{vio}{RGB}{59,18,97}

\begin{document}
\author{Manh-Khang Dao\textsuperscript{}\thanks{\noindent\textsuperscript{}Department of Mathematics, KTH Royal Institute of Technology, 100 44 Stockholm, Sweden. E-mail address: mkdao@kth.se} \@ and  Boualem Djehiche\textsuperscript{}\thanks{\noindent\textsuperscript{}Department of Mathematics, KTH Royal Institute of Technology, 100 44 Stockholm, Sweden. E-mail address: boualem@kth.se}}

\title{Hamilton-Jacobi equations for optimal control on multidimensional
junctions with entry costs\textsuperscript{}\thanks{\noindent\textsuperscript{}{\bf Acknowledgements.} This research was supported by the Swedish Research Council grant (2016-04086). 
This article is a continuation of the works in first author's dissertation
and he appreciate his advisors Yves Achdou, Olivier Ley and Nicoletta Tchou who suggested the topic
to him. 
The authors would like to thank Olivier Ley  for his insightful remarks. The authors are also grateful to the referees for careful reading and valuable comments.}}

\date{\today}

\maketitle

\begin{abstract}
We consider an infinite horizon control problem for dynamics constrained to remain on a multidimensional junction with entry costs. We derive the associated system of Hamilton-Jacobi equations (HJ), prove the comparison principle and that the value function of the optimal control problem is the unique viscosity solution of the HJ system. This is done under the usual strong controllability assumption and also under a weaker condition, coined  'moderate controllability assumption'.
\end{abstract}
\\
\AMSclass{34H05, 35F21, 49L25, 49J15, 49L20, 93C30}

\section{Introduction}
In this paper, we consider an infinite horizon control problem for the dynamics of an agent constrained to remain on a multidimensional junction on $\R^3$, i.e. a union of $N\ge 2$ half-planes $\cP_i$ which share a straight line $\Gamma$, see Figure~\ref{fig:junction}. The controlled dynamics are given by a system of ordinary differential equations, where in each $ \cP_i $ it is given by a drift $f_i(\cdot,\cdot)$ and to which is associated a running cost $\ell_i(\cdot,\cdot)$. Moreover, the agent pays a cost $c_i(\cdot)$ each time it enters the half-plane $\cP_i$ from $\Gamma$. The goal of this work is to study the properties of the value function of this control problem and derive the associated Hamilton-Jacobi equation (HJ) under some regularity conditions on the involved dynamics, running and entry cost functions. 
Although we will not discuss it in this paper, the optimal control problem with exit costs, i.e. instead of paying an entry cost each time the agent enters the half-plane, it pays a cost each time it exits it, can be solved similarly.
Oudet~\cite{Oudet2016_these} considers a similar optimal control problem but without entry or exit costs from the interface to the half-planes.

When the interface $\Gamma$ is reduced to a point, the junction becomes a simple network with one vertex, i.e. a 1-dimensional junction. Optimal control problems (without entry costs) in which the set of admissible states are networks attracted a lot of interest in recent years. Being among the first papers discussing this topic, Achdou, Camilli, Cutr\` i  \& Tchou~\cite{ACCT2013}, derived an HJ equation associated to an infinite horizon optimal control on networks and proposed a suitable notion of viscosity solutions, where the admissible test-functions whose restriction to each edge are $ C^1 $ are applied. Independently and at the same time, Imbert, Monneau \& Zidani~\cite{IMZ2013} proposed an equivalent notion of viscosity solution for studying an HJ approach to junction problems and traffic flows. Both~\cite{ACCT2013} and~\cite{IMZ2013} contain first results on the comparison principle. In the particular case of  eikonal equations on networks, Schieborn \& Camilli~\cite{SC2013} considered a less general notion of viscosity solution. For that later case, Camilli \& Marchi~\cite{CM2013} showed the equivalence between the definitions notion of viscosity solution given in~\cite{ACCT2013,IMZ2013} and \cite{SC2013}. Optimal control on networks with entry costs (and exit costs) has recently been considered by the first author~\cite{Dao2019}.

An important feature of the effect of the entry costs is a possible discontinuity of the value function. Discontinuous solutions of HJ equations have been studied by various authors, see for example Barles~\cite{Barles1993} for general open domains in $ \R^d $, Frankowska \& Mazzola~\cite{FM2013a} for state constraint problems, and in particular Graber, Hermosilla \& Zidani~\cite{GHZ2017} for a class of HJ equations on networks. 

In the case considered in the present work, the effect of entry costs induces a discontinuity of the value function $ \cV $ at the interface $\Gamma$, while it is still continuous on each $ \cP_i \backslash \Gamma $. This allows us to adopt the techniques which apply to the continuous solution case in the works of Barles, Briani \& Chasseigne~\cite{BBC2013} and  Oudet~\cite{Oudet2016_these}, where we split the value function $\cV $ into the collection $\{v_1,\ldots,v_N\}$ of functions, where each $ v_i $ is continuous function defined on $ \cP_i $ and satisfies
\[ 
v_i (x) =
\begin{cases}
\cV(x) & \text {if } x\in \cP_i \backslash \Gamma,\\
\ds \lim_{(\cP_i \backslash \Gamma) \ni z \rightarrow x} \cV(z) & \text{if } x\in \Gamma.
\end{cases}
\]
We note that the existence of the limit in the above formula comes from the fact that the value functions is Lipschitz continuous on the neighborhood of $ \Gamma $ (see Lemma~\ref{lem:continuity}), thanks to the 'strong controllability assumption', which is introduced below.
The first main result of the present work is to show that $(v_1,\ldots,v_N,\cV|_\Gamma)$ is a viscosity solution of the following system
\begin{align}
\lambda v_i (x) + H_i \left( x, \partial v_i (x) \right) = 0, & \quad \text{if }x\in \cP_i \backslash \Gamma,\label{011}\\
\lambda v_i (x) + \max \left\{ H_i^+ \left( x, \partial v_i (x) \right), -\lambda \cV|_\Gamma (x) \right\} = 0, & \quad \text{if }x\in  \Gamma, \label{012}\\
\lambda \cV|_{\Gamma}(x) + \max \left\{ -\lambda \min_{i=1,\ldots,N} \{v_i (x) + c_i(x) \} , H_\Gamma \left( x,\dfrac{\partial \cV|_{\Gamma}}{\partial e_0}  (x) \right) \right\}=0,& \quad \text{if } x\in \Gamma, \label{014}
\end{align}
where $ H_i $ is the Hamiltonian corresponding to the half-plane $ \cP_i $, $ \cV|_\Gamma $ is the restriction of our value function on the interface
and $ H_{\Gamma} $ is the Hamiltonian defined on $\Gamma$. At  $x \in \Gamma $, the definition of the Hamiltonian has to be particular, in order to consider all the possibilities when $x$ is in the neighborhood of $ \Gamma $.
More specifically, 
\begin{itemize}
	\item the term $ H_i^+ \left( x, \partial u_i (x) \right) $ accounts for the situation in which the trajectory does not leave $ \cP_i $,
	\item the term $ \min_{i=1,\ldots,N} \{v_i (x) + c_i(x) \}  $ accounts for situations in which the trajectory enters $\cP_k$ where $ v_k(x) + c_k(x) = \min_{i=1,\ldots,N} \{v_i (x) + c_i(x) \} $,
	\item the term $ H_\Gamma(x,\frac{\partial \cV|_\Gamma}{\partial e_0} (x)) $ accounts for situations in which the trajectory remains on $ \Gamma $.

\end{itemize}

This feature is quite different from the one induced by the effect of entry costs in a network (i.e. when $\Gamma$ is reduced to a point) considered in \cite{Dao2019}, where the value function at the junction point is a constant which is the minimum of the cost when the trajectory stays at the junction point forever and the cost when the trajectory  enters immediately the edge that has the lowest possible cost.

The paper is organized as follows. In Section~\ref{sec2}, we formulate the optimal control problem on a multidimensional junction on $\R^3$ with entry cost. In Section~\ref{A3-cond}, we study the control problem under  the strong controllability condition, where 
we derive the system of HJ equations associated with the optimal control problem, propose a comparison principle, which leads to the well-posedness of \eqref{011}-\eqref{014}, and prove that the value function of the optimal control problem is the unique discontinuous solution of the HJ system. We suggest two different proofs of the comparison principle. The first one is inspired from the work by Lions \& Souganidis \cite{LS2016} and uses arguments from the theory of PDEs, and the second one  uses a blend of arguments from optimal control theory and PDE techniques suggested in \cite{BBC2013,BBC2014,AOT2015} and \cite{Oudet2016_these}. Finally, in Section~\ref{sec8}, the same program is carried out when the strong controllability is replaced by the weaker one that we coin  'moderate controllability near the interface'. The proof  of the comparison principle under the moderate controllability condition is carried on by only using  the PDE techniques provided in Lions \& Souganidis \cite{LS2016}.

%%%%%%%%%%%%%%%%%%%%%%%%%%%%%%%%%%

The results obtained in the present work extend easily to multidimensional junction on $\R^d$, i.e. a union of $N\ge 2$ half-hyperplanes $\cP_i$ which share an affine space $\Gamma$ of dimension $ d-2 $, and to the more general class of ramified sets, i.e. closed and connected subsets of $\R^d$ obtained as the union of embedded manifolds with dimension strictly less than $d$, for which the interfaces are non-intersecting manifolds of dimension $d-2$, see Figure~\ref{fig:a} for example. We do not know whether these results apply to the ramified sets for which interfaces of dimension $ d-2 $ cross each other (see Figure~\ref{fig:b}). Recent results on optimal control and HJ equations on ramified sets  include  Bressan \& Hong \cite{BH2007}, Camilli, Schieborn \& Marchi \cite{CSM2013}, Nakayasu~\cite{Nakayasu2014} and Hermosilla \& Zidani~\cite{HZ2015} and the book of Barles \& Chasseigne \cite{BB2018}.

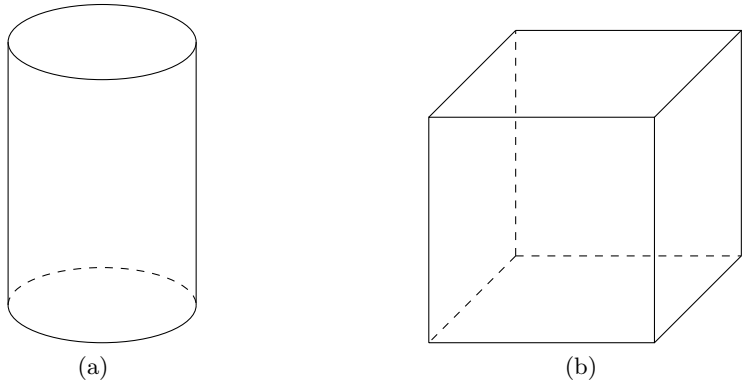
\begin{figure}
	\centering
\subfloat[	\label{fig:a}]{
	\begin{tikzpicture}
	\draw (0,0) ellipse (1.25 and 0.5);
	\draw (-1.25,0) -- (-1.25,-3.5);
	\draw (-1.25,-3.5) arc (180:360:1.25 and 0.5);
	\draw [dashed] (-1.25,-3.5) arc (180:360:1.25 and -0.5);
	\draw (1.25,-3.5) -- (1.25,0);  
	(-1.25,0) -- (-1.25,-3.5) arc (180:360:1.25 and 0.5) -- (1.25,0) arc (0:180:1.25 and -0.5);
	\end{tikzpicture}}\qquad \qquad \qquad \qquad
\subfloat[	\label{fig:b}]{\begin{tikzpicture}
	
	\def\sliceZ{0.8}
	\def\side{3}
	
	% draw slice

	% draw cube
	\draw (\side,0,0) -- (\side,\side,0)  -- (0,\side,0);
	\draw (0,0,\side) -- (\side,0,\side) -- (\side,\side,\side) -- (0,\side,\side) -- (0,0,\side);
	\draw (\side,0,0) -- (\side,0,\side);        
	\draw [dashed] (0,0,0)-- (\side,0,0);
	\draw [dashed] (0,0,0)-- (0,\side,0);
	\draw [dashed] (0,0,0)-- (0,0,\side);
	\draw (\side,\side,0) -- (\side,\side,\side);
	\draw (0,\side,0) -- (0,\side,\side);

	\end{tikzpicture}}
\caption{Example of remafied sets}
\end{figure}

%%%%%%%%%%%%%%%%%%%%%%%%%%%%%%%%%%%%%%%%%%%%%%%%%%%%%%%%%%%%%%%
%%%%%%%%%%%%%%%%%%%%%%%%%%%%%%%%%%%%%%%%%%%%%%%%%%%%%%%%%%%%

\section{Formulation of the control problem on a junction}\label{sec2}
\subsection{The geometry of the state of the system}
Let $\left\{ e_{i}\right\} _{0\le i\le N}$ be distinct unit vectors
in $\mathbb{R}^{3}$ such that  $\ e_i \cdot e_0 =0$
for all $i\in\left\{ 1,\ldots,N\right\} $.
The state of the system is given by the junction $\mathcal{S}$ which is the union of $N$ closed half-planes $\mathcal{P}_{i}=\mathbb{R}e_{0}\times\mathbb{R}^{+}e_{i}$. The half-planes $\mathcal{P}_{i}$
are glued at the straight line $\Gamma:=\mathbb{R}e_{0}$ (see Figure~\ref{fig:junction}).
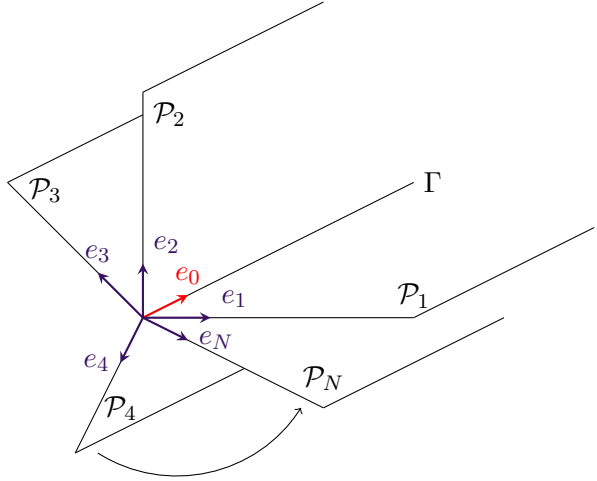
\begin{figure}
	\begin{center}
		\begin{tikzpicture}[scale=0.6, trans/.style={thick,<->,shorten >=2pt,shorten <=2pt,>=stealth} ]
		\draw (0,0) -- (6,3) node[right] {$\Gamma$};
		\draw[->,>=stealth,thick,red,] (0,0) -- (1,0.5)   node[above] {$e_0$};
		%%%%%%%%%%%%%%%%%%%%%%%%%%%%%%%%%%%%%%%%%%%%%%
		\draw (0,0) -- (6,0) node[above] {$\mathcal{P}_1$};
		\draw (6,0) -- (10,2);
		\draw[->,>=stealth,thick,vio,] (0,0) -- (1.5,0)   node[above right] {$e_1$};
		%%%%%%%%%%%%%%%%%%%%%%%%%%%%%%%%%%%%%%%%%%%%%%
		\draw (0,0) -- (0,5) node[below right] {$\mathcal{P}_2$};
		\draw (0,5) -- (4,7);
		\draw[->,>=stealth,thick,vio,] (0,0) -- (0,1.2)   node[above right] {$e_2$};
		%%%%%%%%%%%%%%%%%%%%%%%%%%%%%%%%%%%%%%%%%%%%%%
		\draw (0,0) -- (-3,3) node[below=1mm, right=1.5mm] {$\mathcal{P}_3$};
		\draw (-3,3) -- (0,4.5);
		\draw[->,>=stealth,thick,vio,] (0,0) -- (-1,1)   node[above] {$e_3$};
		%%%%%%%%%%%%%%%%%%%%%%%%%%%%%%%%%%%%%%%%%%%%%%
		\draw (0,0) -- (-1.5,-3) node[above=6mm, right=2.5mm] {$\mathcal{P}_4$};
		\draw (-1.5,-3) -- (9/4,-9/8);
		\draw[->,>=stealth,thick,vio,] (0,0) -- (-0.5,-1)   node[left] {$e_4$};
		%%%%%%%%%%%%%%%%%%%%%%%%%%%%%%%%%%%%%%%%%%%%%%
		\draw (0,0) -- (4,-2) node[above=1mm] {$\mathcal{P}_N$};
		\draw (4,-2) -- (8,0);
		\draw[->,>=stealth,thick,vio,] (0,0) -- (1,-0.5)   node[right] {$e_N$};
		%%%%%%%%%%%%%%%%%%%%%%%%%%%%%%%%%%%%%%%%%%%%%%
		\draw[->] (-1,-3) to [bend right=45] (3.5,-2);
		\end{tikzpicture}
		\caption[Optional caption for list of figures]{The junction $\cS$ in $\R^3$.}
		\label{fig:junction}
	\end{center}
\end{figure}

If $x\in \cS \backslash \Gamma$, there exist unique $i\in \left\{1,\ldots,N\right\}$, $x_i>0$ and $x_0\in \R$ such that
\[
x=x^0 e_0 + x^i e_i =: (x^i,x^0).
\]
Let $x=(x^i,x^0)\in \cP_i$ and $y=(y^j,y^0)\in \cP_j$, the geodesic distance $d(x,y)$ between two points $x,y\in \cS$ is
\[
d(x,y)=\begin{cases}
\left|x-y\right|=\left(\left|x^0-y^0\right|^2+\left|x^i-y^i\right|^2\right)^{\frac{1}{2}}, & \text{if }x,y\in\mathcal{P}_{i},\\
\ds \inf_{z\in\Gamma}\left\{ \left|x-z\right|+\left|z-y\right|\right\}= \left(\left|x^0-y^0\right|^2+\left|x^i+y^j\right|^2\right)^{\frac{1}{2}}, & \text{if }x\in\mathcal{P}_{i} \text{ and } y\in\mathcal{P}_{j}.
\end{cases}
\]

%%%%%%%%%%%%%%%%%%%%%%%%%%%%%%%%%%%%%%%%%

\subsection{The optimal control problem} \label{subsec2}
We consider an infinite horizon optimal control problem which has different dynamics and running costs for each half-plane. For $i=1,\ldots,N$,
\begin{itemize}
\item the set of controls (action set) on $\cP_i$ is denoted by $A_i$,
\item on $\cP_i$ the dynamics of the system is deterministic with associated dynamic $f_i$,
\item the agent has to pay the running cost $\ell_i$ while (s)he is on $\cP_i$.
\end{itemize}

The following conditions, referred to as $[A]$ hereafter,  are our standing assumptions throughout the paper.
\begin{description}
\item{$[A0]$ \textbf{Control sets.}} 
Let $A$ be a metric space (for example $A=\mathbb{R}^d$). For $i=1,\ldots, N$, $ A_i $ is  a nonempty compact subset of $A$ and the sets $A_{i}$ are disjoint.
\item{$[A1]$ \textbf{Dynamics and running costs.}}
For $i=1,\ldots,N$, the functions $\ell_i:\cP_i\times A_{i}\rightarrow\mathbb{R}$ and $ f_{i}: 
\cP_i\times A_{i} \rightarrow \R^3 $
are continuous and bounded by $M$. Moreover, there exists $L>0$ such that 
\[
\left|f_{i}\left(x,a\right)-f_{i}\left(y,a\right)\right|,~ \left|\ell_{i}\left(x,a\right)-\ell_{i}\left(y,a\right)\right|\le L\left|x-y\right|,\quad\mbox{ for all } x,y\in \cP_i, a\in A_{i}.
\]
Hereafter, we will use the notation 
$$
F_{i}\left(x\right):=\left\{ f_{i}\left(x,a\right):a\in A_{i}\right\}.
$$
\item[ ] $\qquad$ {\bf Entry costs.} $\left\{c_1,\ldots,c_N\right\}$ is a set of entry cost functions, where $c_i: \Gamma \rightarrow \R^+$ is Lipschitz continuous and bounded from below by some positive constant $C$.

\item{$[A2]$ \textbf{Convexity of dynamics and costs}.} For $x\in \cP_i$, the following
set
\[
\textsc{FL}_{i}\left(x\right):=\left\{ \left(f_{i}\left(x,a\right),\ell_{i}\left(x,a\right)\right):a\in A_{i}\right\} 
\]
 is non empty, closed and convex.

\end{description}
\begin{rem}
In $[A0]$, the assumption that the set $A_i$ are disjoint is not restrictive since we can always replace $A_i$ by $\tilde{A}_i=A_i\times \left\{i\right\}$. Assumption $[A2]$ is made to avoid the use of relaxed control (see the definition for relaxed control in \cite{BD1997}). Many of these conditions can be weakened at the cost of keeping the presentation of the results easy to follow.
\end{rem}

\subsubsection{Controlled dynamics} 
Let $\cM$ be the closed set given by 
\[ 
\cM := \left\{ (x,a): x\in \cP_i,~a\in A_i \text{ if } x\in \cP_i \backslash \Gamma, \text{ and } a\in \cup_{i=1}^N A_i \text{ if } x\in \Gamma\right\} 
\]
and define the function $f$ on $\cM$ by
\[
f\left(x,a\right)=\begin{cases}
 f_{i}\left(x,a\right), & \quad\mbox{if }x\in \cP_i\backslash \Gamma \mbox{ and }a\in A_{i},\\
f_{i}\left(x,a\right), & \quad\mbox{if }x\in \Gamma \mbox{ and }a\in A_{i}.
\end{cases}
\]
The function $f$ is continuous on $\cM$ since the sets $A_{i}$
are disjoint.
Consider the set $\tilde{F}\left(x\right)$ which contains all the 'possible
speeds' at $x$  defined by
\[
\tilde{F}\left(x\right)=\begin{cases}
F_{i}\left(x\right) & \quad\mbox{if }x\in \cP_i\backslash\Gamma,\\
\bigcup_{i=1}^{N}F_{i}\left(x\right) & \quad\mbox{if \ensuremath{x\in \Gamma}}.
\end{cases}
\]
For $x\in\cS$, the set of admissible trajectories starting
from $x$ is
\[
Y_{x}=\left\{ y_{x}\in Lip\left(\mathbb{R}^{+};\cS\right) \left|
    \begin{array}[c]{l}
\hbox{$\dot{y}_{x}\left(t\right)  \in\tilde{F}\left(y_{x}\left(t\right)\right) \quad\mbox{for a.e. }t>0$}\\
  \hbox{$y_{x}\left(0\right)  =x$.}     
    \end{array} \right.
\right\}.
\]

Thanks to the Filippov implicit function lemma (see \cite{MW1967}), it is shown in \cite[Theorems 3.2.2 and 3.3.1]{Oudet2016_these} that under the respective assumptions $[A3]$ and $[\tilde{A}3]$ below, the set $Y_x$ is not empty. We introduce the set of admissible controlled trajectories starting from $x$
\[
\cT_x=\left\{ \left(y_{x},\alpha\right)\in L_{loc}^{\infty}\left(\mathbb{R}^{+};\mathcal{M}\right):y_{x}\in Lip\left(\mathbb{R}^{+};\cS\right)\mbox{ and }y_{x}\left(t\right)=
x+\int_{0}^{t}f\left(y_{x}\left(s\right),\alpha\left(s\right)\right)ds\right\},
\]
where $ \left(y_{x},\alpha\right)\in L_{loc}^{\infty}\left(\mathbb{R}^{+};\mathcal{M}\right) $ means $ t\mapsto \left(y_{x} (t),\alpha(t)\right)\in L_{loc}^{\infty}\left(\mathbb{R}^{+};\mathcal{M}\right)  $.
We note that if $(y_x,\alpha)\in \cT_x$ then $y_x\in Y_x$. 
Thus, from now on, we will denote $y_x$ by $y_{x,\alpha}$ if $(y_x,\alpha)\in \cT_x$.
By continuity of the trajectory $y_{x,\alpha}$, the set $T^{\Gamma}_{x,\alpha}:=\left\{t\in\R^+: y_{x,\alpha}(t)\in \Gamma\right\}$ containing all the times at which the trajectory stays on $\Gamma$ is closed and therefore, the set $T^{i}_{x,\alpha}:=\left\{t\in\R^+: y_{x,\alpha}(t)\in \cP_i \backslash \Gamma\right\}$ is open.
Consequently, $T^{i}_{x,\alpha}$ is a countable union of disjoint open intervals
\[
T^i_{x,\alpha} =
\begin{cases}
\left[0,\eta_{i0}\right)\cup\bigcup_{k\in K_{i}\subset\mathbb{N^{\star}}}\left(t_{ik},\eta_{ik}\right), & \quad\mbox{if }x\in \cP_i \backslash \Gamma ,\\
\bigcup_{k\in K_{i}\subset\mathbb{N^{\star}}}\left(t_{ik},\eta_{ik}\right), & \quad\mbox{if }x\notin \cP_i \backslash \Gamma ,
\end{cases}
\]
where $K_{i}=\left\{1,\ldots,n\right\}$ if the trajectory $y_{x,\alpha}$ enters $\cP_i$ $n$ times, $K_{i}=\mathbb{N}$ if the trajectory $y_{x,\alpha}$ enters $\cP_i$ infinite times and $ K_i = \emptyset $ if the trajectory never enters $ \cP_i $.

\begin{rem}
From the previous definition, we see that $t_{ik}$ is an entry time in $\cP_i \backslash \Gamma$ and $\eta_{ik}$ is an exit time from $\cP_i \backslash \Gamma$. Hence
\[
x_{ik}:=y_{x,\alpha	}(t_{ik})\in \Gamma,~z_{ik}:=y_{x,\alpha}(\eta_{ik})\in \Gamma.
\]
\end{rem}

 We now define a cost functional and a value function corresponding to the optimal problem.

\subsubsection{Cost functional and value function}\label{v-function-interface}

\begin{defn}
The cost functional associated to the trajectory $(y_x,\alpha)\in \cT_x$ is defined by 
\[
J(x,\alpha)=\int_0^\infty \ell(y_{x} (t),\alpha (t)) e^{-\lambda t} dt + \sum^N_{i=1} \sum_{k\in K_i} c_i (x_{ik}) e^{-\lambda t_{ik}}, 
\]
where $ \lambda>0 $ and the running cost $\ell: \cM \rightarrow \R$ is
\begin{equation}\label{eq:cost-f}
\ell(x,a)=
\begin{cases}
\ell_{i}\left(x,a\right) & \quad\mbox{if \ensuremath{x\in \cP_i\backslash\Gamma}}\mbox{ and }a\in A_{i},\\
\ell_{i}\left(x,a\right) & \quad\mbox{if }x\in \Gamma\mbox{ and }a\in A_{i}.
\end{cases}
\end{equation}

The value function of the infinite horizon optimal control problem is defined by
\begin{equation}\label{eq:value-f}
\cV (x) = \inf_{(y_{x},\alpha) \in \cT_x} J(x,\alpha).
\end{equation}
\end{defn}

\begin{rem}
By the definition of the value function, we are mainly interested in admissible control laws $\alpha$ for which $J(x,\alpha)<+\infty$. In such a case, even if the set $K_i$ may be infinite, it is possible to reorder $\left\{t_{ik},\eta_{ik}:k\in \N \right\}$ such that
\[
t_{i1}<\eta_{i1}<t_{i2}<\eta_{i2}<\ldots<t_{ik}<\eta_{ik}<\ldots,
\]
and
\[
\lim_{k\rightarrow \infty} t_{ik} = \lim_{k\rightarrow \infty} \eta_{ik}=+\infty.
\]
Indeed, because of the positivity of the entry cost functions, if there exists a cluster point,  $J(x,\alpha)$ has to be infinite which leads to a contradiction, since we assumed that $J(x,\alpha)<+\infty$.
This means that the state cannot switch half-planes infinitely many times in finite time, otherwise the cost functional becomes obviously infinite.

\end{rem}

The following example shows that the value function with entry costs can possibly be discontinuous at the interface $\Gamma$.

\begin{ex}
Consider a simple junction $ \cS $ with two half-planes $ \cP_1$ and $ \cP_2 $.
To simplify, we may identify $ \cS \equiv \R^2 $ and 
$\cP_1= \R^+ e_1 \times \R e_0 \equiv (-\infty, 0] \times \R $, $\cP_2 = \R^+ e_2 \times \R e_0 \equiv [0, +\infty) \times \R$ and $\Gamma= \R e_0 \equiv \{0\} \times \R$. The control sets are $A_i=\left\{(a_i,a_0)\in \R^2 : a_0^2+a_i^2\le 1\right\}$ with $i\in\left\{1,2\right\}$. Set
\[
(f(x,a),\ell(x,a))=
\begin{cases}
((a_1,a_0) , 1)& \text{if } x\in \cP_1 \text{ and } a=(a_1,a_0) \in A_1, \\
((a_2,a_0),1-a_2) & \text{if } x\in \cP_2 \text{ and } a=(a_2,a_0) \in A_2,
\end{cases}
\]
and entry costs functions $c_1\equiv C_1$, where $C_1$ is a positive constant and
\[ 
c_2 (\zeta) = 
\begin{cases}
3 - |\zeta|, & \text{ if } \zeta \le 1, \\
 2, & \text{ if } \zeta \ge 1.
\end{cases}
\]
For $x\in \cP_2\backslash \Gamma$, then $\cV(x)=v_2(x)=0$ with optimal strategy which consists of choosing $\alpha\equiv(a_2 = 1, a_0 = 0)$. For $x\in \cP_1$, we can check that 
\begin{itemize}
	\item If $ 2\ge 1/\lambda $, then $ \cV(x) = 1/\lambda $ with optimal control law  $ \alpha \equiv (a_1 = 0, a_0 = 1) $.
	\item If $ 2<1/\lambda $, we consider $ x= (x_1,x_0) \in \cP_i $ in two cases:
	\begin{enumerate}
		\item [Case 1:] If $ |x_1| \ge 1 $, then the optimal control law $ \alpha (t) = (a_1= -1, a_0 =0)$ if $ t\le |x_1| $ and $ \alpha (t) = (a_2 = 1, a_0 = 0) $ if $ t\ge |x_1| $ and
		\[
		\cV(x) = \int^{|x_1|}_0 1 dt +  c_2 ((0,x_1)) e^{-\lambda |x_1|}+0 = \dfrac{1-e^{-\lambda |x_1|} }{\lambda} +2e^{-\lambda |x_1|}.
		\]
		The optimal trajectory starting from $  x = (x_1,x_0) \in \cP_1 $ in case $ |x_1|\ge 1 $ is plotted in red in Figure~\ref{ex}.
		\item [Case 2:] If  $ |x_1| \le 1 $, let $ \tau(x) = \left(|x_1|^2 +|x_0 - 1|^2\right)^{1/2} $
		and $ s: [-1,1] \rightarrow \R $ where $ s(\zeta) =1 $ if $ \zeta \ge 0 $ and $ s(\zeta) =-1 $ if $ \zeta < 0 $. We have the optimal control law
		\[  \alpha(t) = 
		\begin{cases}
		\left(a_1 = -\dfrac{x_1}{\tau (x)}, a_0 = \dfrac{1 - x_0}{\tau (x)}\right), & t\le \tau(x),\\
		(a_2 = 1, a_0 = 0), & t\ge \tau(x).
		\end{cases}
		\] and the value function
		\[		\cV(x)= \int^{\tau(x)}_0 1 dt +  c_2 ( (0,s(x_1)) ) e^{-\lambda}+0= \frac{1-e^{-\lambda \tau(x)}}{\lambda}+ 2 e^{-\lambda \tau(x)}. 
		\]
		The optimal trajectory starting from $  x = (x_1,x_0) \in \cP_1 $ in case $ |x_1|< 1 $ is plotted in blue in Figure~\ref{ex}.
	\end{enumerate}
\begin{figure}
	\begin{center}
		\begin{tikzpicture}[scale=1.3, trans/.style={thick,<->,shorten >=2pt,shorten <=2pt,>=stealth} ]
		\draw (0,-2) -- (0,2) node[right] {$\Gamma$};
		\draw [fill] (0,1)   circle [radius=.04]  node[below right] {$1$}; 
		%		\draw[->,>=stealth,thick,red,] (0,0) -- (0,0.5)   node[right] {$e_0$};
		%%%%%%%%%%%%%%%%%%%%%%%%%%%%%%%%%%%%%%%%%%%%%%
		\draw (2,-2) node[above] {$\mathcal{P}_2$};
		\draw (-2,-2) node[above] {$\mathcal{P}_1$};
		%		\draw [fill] (0,0)   circle [radius=.07]  node[below right] {$0$};
		\draw [fill] (0,-1)   circle [radius=.04]  node[below right] {$-1$};
		\draw [fill] (0,0)   circle [radius=.04]  node[below right] {$0$};
		\draw [fill] (-2,1.5)   circle [radius=.04]  node[below left] {$A$};
		\draw [fill] (-1.5,-0.5)   circle [radius=.04]  node[below left] {$B$};
		%%%%%%%%%%%%%%%%%%%%%%%%%%%%%%%%%%%%%%%%%%%%%%%%%%%%%%%%%%%%
		\draw [dashed, red] (-2,1.5) -- (0,1.5) ;
		\draw [->,>=stealth, dashed, red] (-2,1.5) -- (-1,1.5) ;
		\draw [dashed, red] (0,1.5) -- (2,1.5) ;
		\draw [->,>=stealth, dashed, red] (0,1.5) -- (1,1.5) ;
		%%%%%%%%%%%%%%%%%%%%%%%%%%%%%%%%%%%%%%%%%%%%%%%%%%%%%%%%%%%%
		\draw [dashed, blue] (-1.5,-0.5) -- (0,-1) ;
		\draw [->,>=stealth, dashed, blue] (-1.5,-0.5) -- (-0.75,-0.75) ;
		\draw [dashed, blue] (0,-1) -- (2,-1) ;
		\draw [->,>=stealth, dashed, blue] (0,-1) -- (1,-1) ;
		\end{tikzpicture}
		\caption[Optional caption for list of figures]{The trajectories in case  $2<1/ \lambda$.}
		\label{ex}
	\end{center}
\end{figure}
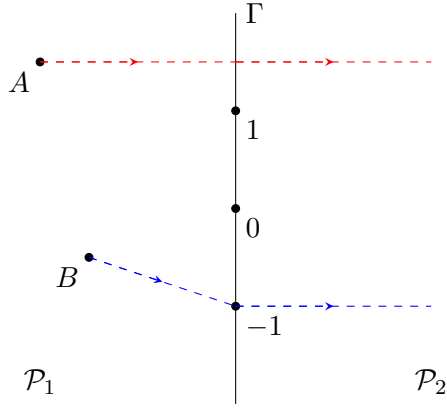	
\end{itemize}

To sum up, there are two cases
\begin{enumerate}
\item If $\ds\inf_{x\in \Gamma}c_2(x) = 2 \ge 1/\lambda$, then
\[
\cV(x)=
\begin{cases}
0 & \text{if } x\in \cP_2 \backslash \Gamma,\\
\dfrac{1}{\lambda} & \text{if } x\in \cP_1.
\end{cases}
\]
The graph of the value function with entry costs satisfying $ \inf_{x\in \Gamma} c_2(x) \ge 1/\lambda$ is plotted in Figure~\ref{figA}.
\item If  $\ds\inf_{x\in \Gamma}c_2(x) = 2 <1/\lambda$, then
\[
\cV(x)=
\begin{cases}
0 & \text{if } x\in \cP_2 \backslash \Gamma,\\
\dfrac{1-e^{-\lambda |x_1| }}{\lambda}+ 2 e^{-\lambda|x_1| } & \text{if } x\in \cP_1\text{ and } |x_1| \ge 1,\\
\dfrac{1-e^{-\lambda \left(|x_1|^2 +|x_0 - 1|^2\right)^{1/2} }}{\lambda}+ 2 e^{-\lambda \left(|x_1|^2 +|x_0 - 1|^2\right)^{1/2} } & \text{if } x\in \cP_1 \text{ and } |x_1| \le 1.
\end{cases}
\]
The graph of the value function in the case $\inf c_2 <1/\lambda$ is plotted in Figure~\ref{figB}.
\end{enumerate}

\begin{figure}
	\centering
	\subfloat[$c_2\ge 1/\lambda$	\label{figA}]{
		\begin{tikzpicture}[scale=0.9,]
		\begin{axis}[colormap/cool,]
		\addplot3[
		surf,
		opacity=0.8,
		samples=50, samples y=30,
		%colormap/whitered,
		domain=-4:0,domain y=-2:2
		%z buffer=sort,
		]
		%{4-4*exp(1/4*\x)};
		{4};
		\addplot3[	
		surf,
		opacity=0.8,
		samples=50, samples y=30,
		%colormap/whitered,
		domain=0:2,domain y=-2:2
		%z buffer=sort,
		]
		{0};
		\end{axis}
		\end{tikzpicture}
	}
	\qquad
	\subfloat[$ c_2 (\zeta) = 3 - |\zeta|$ if $ \zeta \le 1 $ and $ c_2 (\zeta) = 2 $ if $ \zeta \ge 1 $	\label{figB}]{
		\begin{tikzpicture}[scale=0.9,]
		\begin{axis}[colormap/cool,]
		\addplot3[
		surf,
		opacity=0.8,
		samples=50, samples y=30,
		%colormap/whitered,
		domain=-4:0,domain y=-1:0
		%z buffer=sort,
		]
		{4-2*exp (-1/4*sqrt( \x^2 +(1+ \y )^2 ) ) };
		\addplot3[
		surf,
		opacity=0.8,
		samples=50, samples y=30,
		%colormap/whitered,
		domain=-4:0,domain y=0:1
		%z buffer=sort,
		]
		{4-2*exp (-1/4*sqrt( \x^2 +(1- \y )^2 ) ) };
		\addplot3[
		surf,
		opacity=0.8,
		samples=50, samples y=30,
		%colormap/whitered,
		domain=0:2,domain y=-2:2
		%z buffer=sort,
		]
		{0};
		\addplot3[
		surf,
		opacity=0.8,
		samples=50, samples y=30,
		%colormap/whitered,
		domain=-4:0,domain y=-2:-1
		%z buffer=sort,
		]
		{4-2*exp(1/4*\x) };
		\addplot3[
		surf,
		opacity=0.8,
		samples=50, samples y=30,
		%colormap/whitered,
		domain=-4:0,domain y=1:2
		%z buffer=sort,
		]
		{4-2*exp (1/4*\x ) };
		\end{axis}
		\end{tikzpicture}}
	\caption{The value function $\cV$ in two cases}\label{fig2}
\end{figure}
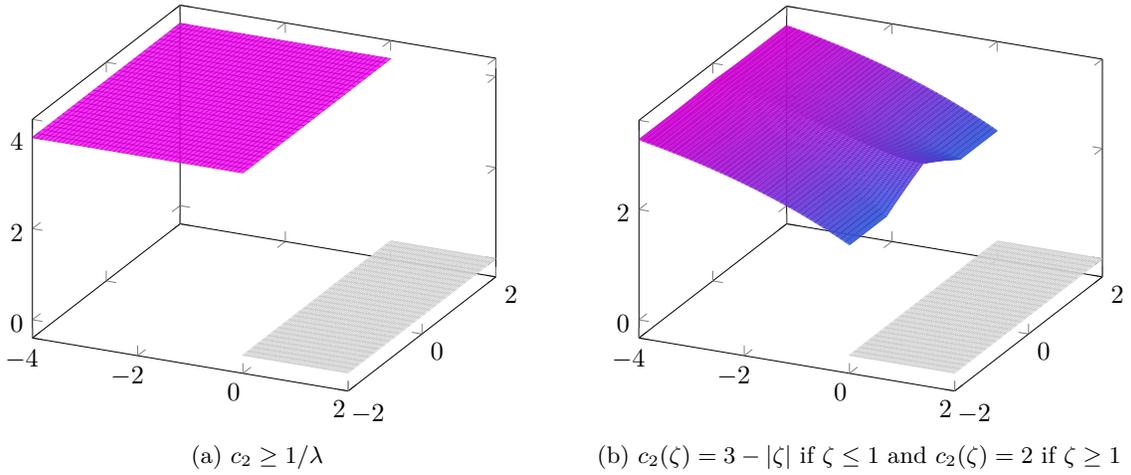

\end{ex}

%%%%%%%%%%%%%%%%%%%%%%%%%%%%%%%%%%%%%%%%%%%%%%%%%%%%%%%%%%%%%%%%%%%%%%%%%%%%%%%%%%%%%%%%%%%%%%%%%%
%%%%%%%%%%%%%%%%%%%%%%%%%%%%%%%%%%%%%%%%%%%%%%%%%%%%%Strong controllability%%%%%%%%%%%%%%%%%%%%%%%%

\section{Hamilton-Jacobi system under strong controllability condition near the interface}\label{A3-cond}
In this section we derive the Hamilton-Jacobi system (HJ) associated with the above optimal control problem and prove that the value function given by \eqref{eq:value-f} is the unique viscosity solution of that  (HJ) system, under the following condition:

\begin{description}
\item{$\left[A3\right]$ (\textbf{Strong controllability})} There exists a real
number $\delta>0$ such that for any $i=1,\ldots,N$ and for all $x\in \Gamma$,
\[
B(0,\delta)\cap (\R e_0 \times \R e_i)\subset F_{i}\left(x\right).
\]
\end{description}

\begin{rem}
If $x$ is close to $\Gamma$, we can use $[A3]$ to obtain the coercivity of the Hamiltonian which will be needed in Lemma \ref{lem:continuity-sub} below to prove the Lipschitz continuity of the viscosity subsolution of the HJ system.
\end{rem}
Hereafter we will denote by $B(\Gamma,\rho),\,\rho>0,$ the set
$$
B(\Gamma,\rho):=\left\{x: \inf_{z\in\Gamma} |x-z|< \rho \right\}.
$$

\begin{lem}\label{lem:strong control}
Under Assumptions $[A1]$ and $[A3]$, there exist two positive numbers $r_0$ and $C$ such that for all $x_1,x_2\in B(\Gamma,r_0)$, there exist $(y_{x_1}, \alpha_{x_1,x_2})\in \cT_{x_1}$ and $\tau_{x_1,x_2}<C d(x_1,x_2)$ such that $y_{x_1}(\tau	_{x_1,x_2})=x_2$.
\end{lem}

\begin{proof}
The proof is classical and similar to the one in \cite{Dao2019}, so we skip it.
\end{proof}
%%%%%%%%%%%%%%%%%%%%%%%%%%%%%%%%%%%%%%%%%
%%%%%%%%%%%%%%%%%%%%%%%%%%%%%%%%%%%%%%%%%

\subsection{Value function on the interface}\label{sec3}

\begin{lem}\label{lem:continuity}
Under Assumptions $[A]$ and $[A3]$, for all $i\in\{1,\ldots,N\}$, $\cV|_{\cP_i \backslash \Gamma}$ is continuous. Moreover, there exists $\varepsilon>0$  such that $\cV|_{\cP_i \backslash \Gamma}$ is Lipschitz continuous in $B(\Gamma,\varepsilon)\cap \cP_i \backslash \Gamma$. Therefore, it is possible to extend $\cV|_{\cP_i \backslash \Gamma}$ to the interface $\Gamma$ and from now on, we use the following notation
\begin{equation}\label{eq:v-i}
v_i(x)=
\begin{cases}
\cV|_{\cP_i}(x), &\text{if } x\in \cP_i \backslash \Gamma,\\
\ds \lim_{ ( \cP_i \backslash \Gamma) \ni z \rightarrow x} \cV|_{\cP_i}(z), &\text{if } x\in  \Gamma.
\end{cases}
\end{equation}
\end{lem}
\begin{proof}
This lemma is a consequence of Lemma~\ref{lem:strong control}, see~\cite{ACCT2011} and \cite{Dao2019} for more details.
\end{proof}

For $x\in \Gamma$, we set
\begin{equation}\label{eq:FL}
\textsc{FL}_\Gamma (x) := \bigcup_{i=1}^N \left( \textsc{FL}_i(x) \cap (\R e_0 \times \R) \right)
\end{equation}
and 
\begin{equation}\label{eq:A-Gamma}
A^{\Gamma}_{i}(x)=\left\{a\in A_i : f_i(x,a)\cdot e_i =0 \right\},\quad i=1,\ldots N,
\end{equation}
where $ \textsc{FL}_i(x) $ is defined in Assumption $ [A2] $.
Let us define a viscosity solution of the \textit{switching} Hamilton-Jacobi equation on the interface~$\Gamma$:
\begin{equation}\label{eq:HJG}
\lambda u_{\Gamma}(x) + \max \left\{ -\lambda \min_{i=1,\ldots,N} \{v_i (x) + c_i(x) \} , H_\Gamma \left( x,\dfrac{\partial u_\Gamma}{\partial e_0}  (x) \right) \right\}=0,\quad x\in \Gamma,
\end{equation}
where $H_\Gamma$ is the Hamiltonian on $\Gamma$ defined by
\begin{equation}\label{eq:H-G}
H_\Gamma (x,p)= \max_{i=1,\ldots, N} \max_{a \in A^{\Gamma}_{i}(x)} \left\{  -(f_i(x,a) \cdot e_0) p - \ell_i(x,a) \right\} = \max_{(\zeta,\xi)\in \textsc{FL}_\Gamma (x)} \left\{  - (\zeta \cdot e_0 ) p - \xi \right\}.
\end{equation}
\begin{defn}
An upper (resp. lower) semi-continuous $u_\Gamma : \Gamma \rightarrow \R$ is a viscosity subsolution (resp. supersolution) of \eqref{eq:HJG} if for any $x\in \Gamma$, any $\varphi \in C^1 ( \Gamma )$ such that $u_\Gamma - \varphi$ has a local maximum (resp. minimum) point at $x$, then
\[
\lambda u_{\Gamma}(x) + \max \left\{ -\lambda \min_{i=1,\ldots,N} \{v_i (x) + c_i(x) \} , H_\Gamma \left(x,\dfrac{\partial \varphi}{\partial e_0}  (x)\right)\right\}\le0 \quad (\text{resp.}\ge 0 ).
\]
The continuous function $u_\Gamma : \Gamma \rightarrow \R$ is called viscosity solution of \eqref{eq:HJG} if it is both viscosity sub and supersolution of \eqref{eq:HJG}.
\end{defn}

We have the following characterization of the value function $\cV$ on the interface.

\begin{thm}\label{thm1}

Under Assumptions $[A]$ and $[A3]$, the restriction of value function $\cV$ on the interface $\Gamma$, $\cV|_{\Gamma}$, is a viscosity solution of \eqref{eq:HJG}. 

\end{thm}

%We postpone the proof of uniqueness to the appendix. 
The proof of Theorem~\ref{thm1} is made in several steps. The first step is to prove that $\cV|_\Gamma$ is a viscosity solution of an HJ equation with an extended definition of the Hamiltonian on $\Gamma$. For that, we consider the following larger relaxed vector field: for $x\in\Gamma$,
\[
f\ell_\Gamma (x)=\left\{ (\eta,\mu)\in \R e_0 \times \R:     \left. \begin{array}[c]{l}
\hbox{$\exists \{y_{x,n}, \alpha_n \}_{n\in \N}$}\\
\hbox{$ (y_{x,n}, \alpha_n)\in \cT_x$}\\
  \hbox{$\exists \{t_n \}_{n\in \N}$}     
    \end{array} \right.\left|
    \begin{array}[c]{l}
\hbox{$t_n\rightarrow 0^+,~y_{x,n}(t)\in \Gamma \text{ for all } t\in [0,t_n] $}\\
 \hbox{$\ds \eta =\lim_{n\rightarrow +\infty} t_n^{-1}\int_0^{t_n} f(y_{x,n} (t), \alpha_n(t)) dt$}\\
  \hbox{$\ds \mu =\lim_{n\rightarrow +\infty} t_n^{-1}\int_0^{t_n} \ell(y_{x,n} (t), \alpha_n(t)) dt$}     
    \end{array} \right.
\right\}.
\] 
We have 
\begin{lem}\label{lem:fl=FL}
For any function $\varphi\in C^1 (\Gamma)$ and $x\in \Gamma$,
\[
\sup_{(\zeta,\xi)\in f\ell_\Gamma (x)} \left\{  -(\zeta\cdot e_0 ) \dfrac{\partial \varphi}{\partial e_0} (x) - \xi \right\} = \sup_{(\zeta,\xi)\in \textsc{FL}_\Gamma (x)} \left\{ (-\zeta \cdot e_0) \dfrac{\partial \varphi}{\partial e_0} (x) - \xi \right\}. 
\]
\end{lem}
\begin{proof}
See Appendix.
\end{proof} 

The second step consists of proving the following lemma.
\begin{lem}\label{lem2.2}
The restriction of the value function $\cV$ on the interface $\Gamma$, $\cV|_{\Gamma}$ satisfies
\begin{equation}\label{eq:HJG2}
\lambda \cV|_\Gamma (x) + H_\Gamma \left( x, \dfrac{\partial \cV|_\Gamma}{\partial e_0} (x) \right) \le  0, \quad x\in \Gamma,
\end{equation}
in the viscosity sense.
\end{lem}

\begin{proof}
Let $x \in\Gamma$ and $\varphi\in C^1 (\Gamma)$ such that $\cV|_\Gamma-\varphi$ has a maximum at $x$,  i.e.
\begin{equation}\label{491}
\varphi(x) - \varphi (z) \le \cV|_\Gamma (x) - \cV|_\Gamma (z), \quad \text{for all }z\in \Gamma.
\end{equation}
From Lemma~\ref{lem:fl=FL}, it suffices to prove that
\begin{equation}\label{490}
\lambda V|_\Gamma (x) + \sup_{(\zeta,\xi)\in f\ell_\Gamma (x)} \left\{ (-\zeta \cdot e_0) \dfrac{\partial \varphi}{\partial e_0} (x) -\xi \right\} \le 0.
\end{equation}
Let $(\zeta,\xi) \in f\ell_\Gamma (x) $, there exist $(y_{x,n},\alpha_n)\in \cT_x$ and $t_n \rightarrow 0^+$ such that $y_{x,n} (t) \in \Gamma$ for all $t\le t_n$ and
\begin{align*}
\zeta & =\lim_{n\rightarrow+\infty}\dfrac{1}{t_{n}}\int_{0}^{t_{n}}f\left(y_{x,n}\left(t\right),\alpha_{n}\left(t\right)\right)dt=\lim_{n\rightarrow\infty}\dfrac{y_{x,n}\left(t_{n}\right)-x}{t_{n}},\\
\xi & =\lim_{n\rightarrow+\infty}\dfrac{1}{t_{n}}\int_{0}^{t_{n}}\ell\left(y_{x,n}\left(t\right),\alpha_{n}\left(t\right)\right)dt.
\end{align*}
According to \eqref{491} and the dynamic programming principle, for all $n\in \N$,
\begin{align*}
\varphi\left(x\right)-\varphi\left(y_{x,n}(t_{n})\right) & \le\mathcal{V}|_{\Gamma}\left(x\right)-\mathcal{V}|_{\Gamma}\left(y_{x,n}(t_{n})\right)\\
 & \le\int_{0}^{t_{n}}\ell\left(y_{x,n}\left(t\right), \alpha_n(t) \right)e^{-\lambda t}dt+\mathcal{V}|_{\Gamma}\left(y_{x,n}\left(t_{n}\right)\right) (e^{-\lambda t_{n}}-1 ).
\end{align*}
Dividing both sides by $t_n$, the goal is to take the limit as $n$ tends to $\infty$.
On the one hand, we have
\begin{equation}\label{492}
\lim_{n\rightarrow+\infty} \left(\dfrac{1}{t_n} \int_0^{t_n} \ell \left(y_{x,n}(t), \alpha_n(t) \right) e^{-\lambda t }dt + \dfrac{\cV|_\Gamma (y_{x,n}(t_n)) (e^{-\lambda t_n }-1)}{t_n}  \right) = \xi - \lambda V|_{\Gamma} (x).
\end{equation}
On the other hand, since $y_{x,n} (t_n) = x + t_n (\zeta + o(1)e_0 )$, we obtain 
\begin{equation}\label{493}
\lim_{n\rightarrow +\infty} \dfrac{\varphi (x) - \varphi (y_{x,n} (t_n))}{t_n}=-(\zeta \cdot e_0) \dfrac{\partial \varphi}{\partial e_0} (x) .
\end{equation}
Hence, in view of \eqref{492} and \eqref{493}, we have
\begin{equation}\label{494}
\lambda \cV|_\Gamma (x) - (\zeta \cdot e_0)  \dfrac{\partial \varphi}{\partial e_0} (x) - \xi \le 0.
\end{equation}
Thus \eqref{494} holds for any $(\zeta,\xi) \in f\ell_\Gamma (x) $ and therefore \eqref{490} holds.
\end{proof}

\begin{lem}\label{lem2.1}
Under Assumptions $[A]$ and $[A3]$, for all $x\in \Gamma$, 
\[
\max_{i=1,\ldots,N} \{ v_i (x) \} \le \cV|_\Gamma (x) \le \min_{i=1,\ldots,N} \{v_i (x) +c_i (x) \}.
\]
\end{lem}

\begin{proof}
Let $i\in \{1,\ldots,N\}$, $x\in \Gamma$ and $z\in \cP_i \backslash \Gamma$ such that $|x-z|$ is small.
It suffices to prove (a) $ v_i (x) \le \cV|_\Gamma (x) $ and (b) $ \cV|_\Gamma (x)\le v_i(x)+c_i (x) $.

\begin{enumerate}
	\item [(a)] Consider any control law $\alpha$ such that $(y_{x}, \alpha) \in \cT_x$.
	Let $\alpha_{z,x}$ be a control law which connects $z$ to $x$ (which exists thanks to Lemma~\ref{lem:strong control}) and consider the control law
	\[
	\hat{\alpha}(x) =
	\begin{cases}
	\alpha_{z,x} (s) & \text{if } s\le \tau_{z,x},\\
	\alpha(s-\tau_{z,x}) & \text{if } s>\tau_{z,x}.
	\end{cases}
	\] 
	This means that the trajectory goes from $z$ to $x$ with the control law $\alpha_{z,x}$ and then proceeds with the control law $\alpha$. 
	Therefore,
	\[
	\cV (z) = v_i (z) \le J(z,\hat{\alpha}) = \int_0^{\tau_{z,x}} \ell_i \left(y_{z,\hat{\alpha}} (s), \hat{\alpha}(s) \right) e^{-\lambda s} ds  + e^{-\lambda \tau_{z,x}}  J(x,\alpha). 
	\]
	Since $\alpha$ is chosen arbitrarily and $\ell_i$ is bounded by $M$, we obtain
	\[
	v_i(z) \le M \tau_{z,x} + e^{-\lambda \tau_{z,x}} \cV(x). 
	\]
	Let $z$ tend to $x$ (then $\tau_{z,x}$ tends to $0$ by Lemma~\ref{lem:strong control}), we conclude $v_i(x)\le \cV(x)$. 
	\item [(b)] Consider any control law $\alpha_z$  such that $(y_{z}, \alpha_z)\in \cT_z$ and use
	Lemma~\ref{lem:strong control} to pick a control law $\alpha_{x,z}$ connecting $x$ to $z$. Consider the control law 
	\[
	\hat{\alpha}(x)=
	\begin{cases}
	\alpha_{x,z}(s) & \text{if } s\le \tau_{x,z},\\
	\alpha_z (s-\tau_{x,z}) & \text{if } s > \tau_{x,z},
	\end{cases}
	\]
	for which the trajectory $y_{x,\hat{\alpha}}$ goes from $x$ to $z$ using the control law $\alpha_{x,z}$ and then proceeds with the control law $\alpha_z$. Therefore,
	\[
	\cV(x) \le J(x,\hat{\alpha}) = c_i (x) + \int^{\tau_{x,z}}_0 \ell_i \left( y_{x,\hat{\alpha}} (s), \hat{\alpha}(s) \right) e^{-\lambda s }ds + e^{-\lambda \tau_{x,z}} J(z,\alpha_z).
	\]
	Since $\alpha_z$ is chosen arbitrarily and $\ell_i$ is bounded by $M$, we obtain
	\[
	\cV(x) \le c_i(x) + M\tau_{x,z}+ e^{-\lambda \tau_{x,z}} v_i(z).
	\]
	Let $z$ tend $x$ (then $\tau_{x,z}$ tends to $0$, by Lemma~\ref{lem:strong control}), we conclude $\cV(x)\le c_i(x) + v_i (x)$.
\end{enumerate}

\end{proof}

From Lemma~\ref{lem2.2} and Lemma~\ref{lem2.1}, we conclude that $\cV|_{\Gamma}$ is a viscosity subsolution of~\eqref{eq:HJG}. 
The last step of of the proof of Theorem~\ref{thm1} is to prove that $\cV|_{\Gamma}$ is a viscosity supersolution of~\eqref{eq:HJG}.
\begin{lem}\label{lem2.3}
The restriction of the value function $\cV$ on the interface $\Gamma$, $\cV|_{\Gamma}$ satisfies
\begin{equation*}
\lambda \cV|_\Gamma (x) + \max \left\{   -\lambda \min_{i=1,\ldots,N} \{v_i(x) +c_i(x) \}, 
H_\Gamma \left( x, \dfrac{\partial \cV|_\Gamma}{\partial e_0} (x) \right) \right\} \ge  0, \quad x\in \Gamma,
\end{equation*}
in the viscosity sense.
\end{lem}

\begin{proof}
Let $ x\in \Gamma $ and assume that
\begin{equation}\label{eq:cV<ci}
\cV(x) < \min_{i=1,\ldots,N} \{ v_i(x) + c_i(x) \},\quad \text{for all } n \in \N,
\end{equation}
it suffices to prove that $\cV(x)$ satisfies 
\[
\lambda \cV|_\Gamma (x) + H_\Gamma \left( x, \dfrac{\partial \cV|_\Gamma}{\partial e_0} (x) \right) \ge 0,   
\]
in the viscosity sense. 
Let $ \{\varepsilon_n\} $ be a sequence which tends to $ 0 $. For any $n$, let $\alpha_n$ be an $\varepsilon_n$-optimal control, i.e. $\cV(x) + \varepsilon_n > J(x,\alpha_n)$, and $ \tau_n $ be the first time the trajectory  $y_{x,\alpha_n}$ leaves $\Gamma$, i.e.
\[
\tau_n:= \inf_{i=1,\ldots,N} T_{x,\alpha_n}^i, \text{ where } T_{x,\alpha_n}^i := \{t\in \R^+ :  y_{x,\alpha_n} (t) \in \cP_i \backslash \Gamma \}.
\]
We note that $\tau_n$ is possibly $+\infty$, in which case the trajectory $y_{x,\alpha_n}$ stays on $\Gamma$ for all $s\in [0,+\infty)$.
We consider the two following cases:
\begin{description}
\item{\emph{Case 1:}} There exists a subsequence of $\{\tau_n\}$ (which is still denoted $\{\tau_n\}$) such that $\tau_n \rightarrow 0$ as $n\rightarrow +\infty$ and at time $\tau_n$ the trajectory enters $\cP_{i_0}$, for some $i_0 \in \{1,\ldots,N \}$.
This implies
\begin{align*}
 \cV(x) + \varepsilon_n & >  J(x,\alpha_n) \\
& =   \int_0^{\tau_n} \ell \left( y_{x,\alpha_n} (s), \alpha_n (s) \right) e^{-\lambda s} ds + c_{i_0}(y_{x,\alpha_n} (\tau_n)) e^{-\lambda \tau_n} + v_{i_0} (y_{x,\alpha_n} (\tau_n)) e^{-\lambda \tau_n}.
\end{align*}
Since $ \ell $ is bounded by $ M $, sending $n$ to $+\infty$, yields
\[
\cV (x) \ge c_{i_0}(x) +v_{i_0} (x),
\]
which leads to a contradiction to \eqref{eq:cV<ci}.

\item{\emph{Case 2:}} There exist a subsequence of $\{\tau_n\}$ (which is still denoted $\{\tau_n\}$) and a positive constant $C$ such that $\tau_n >C$.
This means that from $0$ to $C$, the trajectory $y_{x,\alpha_n}$ still remains in $\Gamma$. Thus, for all  $\tau\in [0,C],$
\begin{align}\label{497}
\cV|_\Gamma (x)  + \varepsilon_n & \ge  \int_{0}^{\tau}\ell\left(y_{x,n}\left(t\right), \alpha_n(t) \right)e^{-\lambda t}dt+\mathcal{V}|_{\Gamma}\left(y_{x,n}\left(\tau\right)\right)  e^{-\lambda \tau}\nonumber\\
& \ge  \int_{0}^{\tau}\ell\left(y_{x,n}\left(t\right), \alpha_n(t) \right)dt+\mathcal{V}|_{\Gamma}\left(y_{x,n}\left(\tau\right)\right)  e^{-\lambda \tau}+ o(\tau),
\end{align}
where $o(\tau)/\tau \rightarrow 0$ as $\tau \rightarrow 0$ and the last inequality is obtained by using the boundedness of $\ell$.
Let $\varphi\in C^1 (\Gamma)$ such that $\cV|_\Gamma-\varphi$ has a minimum on $ \Gamma $ at $x$, i.e.
\begin{equation}\label{498}
\varphi(x) - \varphi (z) \ge \cV|_\Gamma (x) - \cV|_\Gamma (z), \quad \text{for all }z\in \Gamma.
\end{equation}
From Lemma~\ref{lem:fl=FL}, it suffices to prove that
\begin{equation}\label{400}
\lambda \cV|_\Gamma (x) + \max_{(\zeta,\xi)\in f\ell_\Gamma (x) } \left\{ - (\zeta \cdot e_0)  \dfrac{\partial \varphi}{\partial e_0} (x) -\xi \right\} \ge 0. 
\end{equation}
Since $ \lim_{n\rightarrow \infty} \varepsilon_n =0 $, it is possible to choose a sequence $ \{t_n\} $ such that $0<t_n <C $ and $ \varepsilon_n / t_n \rightarrow 0 $ as $ n \rightarrow \infty $.
Thus from \eqref{497} and \eqref{498}, we obtain
\begin{equation}\label{499}
\dfrac{\varphi (x) - \varphi(y_{x,n}(t_n))}{t_n} -\dfrac{1}{t_n} \int_0^{t_n} \ell(y_{x,n} (t), \alpha_n (t) ) dt + \dfrac{1-e^{-\lambda t_n}  }{t_n}\cV|_\Gamma (y_{x,n} (t_n) ) \ge -\dfrac{\varepsilon_n}{t_n} + o(1). 
\end{equation}
Since $f$ and $\ell$ are bounded, then the sequence  $\left\{ \frac{y_{x,n} (t_n) - x }{t_n}, \frac{1}{t_n} \int_0^{t_n} \ell(y_{x,n} (t), \alpha_n (t) ) dt  \right\}$ is bounded in $\Gamma \times \R$. Therefore, we can extract a subsequence of this sequence which converges to $(\bar{\zeta},\bar{\xi})$ as $n\rightarrow +\infty$.
Obviously, we have $\left( \bar{\zeta},\bar{\xi} \right) \in f\ell_\Gamma (x)$.
Hence, sending $n$ to $\infty$ in \eqref{499}, we obtain
\[
\lambda \cV|_\Gamma (x) - (\bar{\zeta}\cdot e_0) \dfrac{\partial \varphi}{\partial e_0} (x) -\bar{\xi}  \ge 0,
\]
and thus \eqref{400} holds.

\end{description}

\end{proof}

\subsection{The Hamilton-Jacobi system and viscosity solutions}\label{sec4}
\subsubsection{Admissible test-functions}

\begin{defn}
A function $\varphi:\cP_1 \times\ldots\times \cP_N \times \Gamma \rightarrow \R^N$ is an admissible test-function if there exist $(\varphi_1,\ldots,\varphi_N, \varphi_\Gamma)$, $\varphi_i\in C^1(\cP_i)$ and $ \varphi_\Gamma \in C^1(\Gamma) $, such that $\varphi(x_1,\ldots,x_N,x_\Gamma)=(\varphi_1 (x_1), \ldots, \varphi_N (x_N), \varphi_\Gamma (x_\Gamma) )$. The set of admissible test-function is denoted by $\mathcal{R}(\cS)$.
\end{defn}

\subsubsection{Hamilton-Jacobi system}

Define the Hamiltonian $H_i : \cP_i \times (\R e_0 \times \R e_i) \rightarrow \R$ by  
\[
H_i(x,p)=\max_{a\in A_i} \left\{ -f_i(x,a)\cdot p - \ell_i (x,a) \right\}
\]
and the Hamiltonian $H_i^+ : \Gamma \times (\R e_0 \times \R e_i) \rightarrow \R$ by
\[
H_i^+(x,p)=\max_{a\in A_i^+(x) } \left\{ -f_i(x,a)\cdot p - \ell_i (x,a) \right\},
\]
where $A_i^+ (x) = \{a\in A_i : f_i(x,a)\cdot e_i \ge 0 \}$ and consider the following Hamilton-Jacobi system
\begin{equation}\label{eq:HJmain}
\left\{\begin{array}{rr}
\lambda u_{i}\left(x\right)+H_{i}\left(x, \partial u_i (x) \right)=0, & \text{if } x\in \cP_{i}\backslash \Gamma,\\
 \lambda u_{i} (x)+\max\left\{H_{i}^{+}\left(x,\partial u_i (x) \right), -\lambda u_\Gamma (x) \right\} =0, & \text{if } x\in \Gamma, \\
 \ds \lambda u_{\Gamma}(x) + \max \left\{ -\lambda \min_{i=1,\ldots,N} \{u_i (x) + c_i(x) \} , H_\Gamma \left( x,\dfrac{\partial u_\Gamma}{\partial e_0}  (x) \right) \right\}=0,\\
 i\in \{1,\ldots, N \}, 
\end{array}\right.
\end{equation}
and its viscosity solution $U:=\left(u_{1},\ldots,u_{N}, u_\Gamma\right)$. 

\begin{defn}[Viscosity solution with entry costs]\label{def:vis-sol}

$ $
\begin{itemize}
	\item A function $U:=\left(u_{1},\ldots,u_{N},u_\Gamma\right)$
	where $u_{i}\in USC\left(\cP_{i};\mathbb{R}\right)$ for all $i\in \{1,\ldots, N \}$ and $ u_\Gamma \in C(\Gamma,\R) $, is
	called a \emph{viscosity subsolution} of~\eqref{eq:HJmain}
	if for any $\left(\varphi_{1},\ldots,\varphi_{N}, \varphi_\Gamma\right)\in\mathcal{R}(\cS)$,
	any $i\in\{1,\ldots, N \}$ and any $x_{i}\in \cP_{i}$, $ x\in \Gamma $ such that $u_{i}-\varphi_{i}$
	has a \emph{local maximum point} on $\cP_{i}$ at $x_{i}$ and $ u_\Gamma - \varphi_\Gamma $ has a \emph{local maximum point} on $\Gamma$ at $x$, then
	\begin{align*}
	\lambda u_{i}\left(x_i\right)+H_{i}\left(x_i, \partial \varphi_i (x_i) \right) \le 0, & \quad\text{if } x_i\in \cP_{i}\backslash \Gamma,\\
	\lambda u_{i} (x_i)+\max\left\{H_{i}^{+}\left(x_i,\partial \varphi_i (x_i) \right), -\lambda u_\Gamma (x_i) \right\} \le 0, & \quad \text{if } x_i\in \Gamma,\\
	\ds \lambda u_{\Gamma}(x) + \max \left\{ -\lambda \min_{i=1,\ldots,N} \{u_i (x) + c_i(x) \} , H_\Gamma \left( x,\dfrac{\partial \varphi_\Gamma}{\partial e_0}  (x) \right) \right\} \le 0, 
	\end{align*}
	\item A function $  U:=\left(u_{1},\ldots,u_{N}, u_\Gamma\right)$
	where $u_{i}\in LSC\left(\cP_{i};\mathbb{R}\right)$ for all $i\in \{1,\ldots, N \}$ and $ u_\Gamma \in C(\Gamma;\R) $, is
	called a \emph{viscosity supersolution} of~\eqref{eq:HJmain}
	if for any $\left(\varphi_{1},\ldots,\varphi_{N}, \varphi_\Gamma\right)\in\mathcal{R}(\cS)$,
	any $i\in\{1,\ldots, N \}$ and any $x_{i}\in \cP_{i}$ and $ x\in \Gamma $ such that $u_{i}-\varphi_{i}$
	has a \emph{local minimum point} on $\cP_{i}$ at $x_{i}$  and $ u_\Gamma - \varphi_\Gamma $ has a \emph{local minimum point} on $\Gamma$ at $x$, then
	\begin{align*}
	\lambda u_{i}\left(x_i\right)+H_{i}\left(x_i, \partial \varphi_i (x_i) \right) \ge 0 & \quad \text{if } x_i\in \cP_{i}\backslash \Gamma,\\
	\lambda u_{i} (x_i)+\max\left\{H_{i}^{+}\left(x_i,\partial \varphi_i (x_i) \right), -\lambda u_\Gamma (x_i) \right\} \ge 0 & \quad \text{if } x_i\in \Gamma,\\
	\ds \lambda u_{\Gamma}(x) + \max \left\{ -\lambda \min_{i=1,\ldots,N} \{u_i (x) + c_i(x) \} , H_\Gamma \left( x,\dfrac{\partial \varphi_\Gamma}{\partial e_0}  (x) \right) \right\} \ge 0,
	\end{align*} 
	\item A functions $  U:=\left(u_{1},\ldots,u_{N}, u_\Gamma\right)$
	where $u_{i}\in C\left(\cP_{i};\mathbb{R}\right)$ for all $i\in \{1,\ldots,N\}$ and $ u_\Gamma \in C(\Gamma;\R) $, is called
	a \emph{viscosity solution} of~\eqref{eq:HJmain}
	if it is both a viscosity subsolution and a viscosity supersolution
	of~\eqref{eq:HJmain}.
\end{itemize}

\end{defn}
%%%%%%%%%%%%%%%%%%%%%%%%%%%%%%%%%%%%%%%%%
%%%%%%%%%%%%%%%%%%%%%%%%%%%%%%%%%%%%%%%%%
\subsection{Relations between the value function and the HJ system}\label{sec5}
In this section, we wish to prove that
\begin{thm}
	Under Assumptions [A] and [A3], $ V:=(v_1,\ldots,v_N, \cV|_\Gamma )$ is a viscosity solution of~\eqref{eq:HJmain}, where the functions $v_i$ are defined in \eqref{eq:v-i}.
\end{thm}

\begin{proof}
	
	By Theorem~\ref{thm1}, $ \cV|_\Gamma $ is a viscosity solution of~\eqref{eq:HJG}. 
	Furthermore, if $ x \in \cP_i \backslash \Gamma $ , for any $ i \in \{1,\ldots,N \} $ and $ (y_x,\alpha) \in \cT_{x} $, there exists a time $ \tau $ small enough so that $ y_{x,\alpha} (t) \in \cP_i \backslash \Gamma$ for $ 0\le t \le \tau $. Thus, the proof in this case is classical by using dynamic programming principle (see \cite{BD1997,Barles1994}) and we do not detail it.
	Now assume $ x\in \Gamma $, we shall prove that for all $i\in \{1,\ldots,N\}$, the function $v_i$ satisfies
	\begin{equation}\label{eq:HJG3}
	\lambda v_{i} (x)+\max\left\{H_{i}^{+}\left(x,\partial v_i (x) \right),-\lambda \cV|_\Gamma (x) \right\} =0,\quad \text{on }  \Gamma, 
	\end{equation}
	in the viscosity sense. 
	The proof of this case is a consequence of Lemma~\ref{lem:value-sub-sol} and Lemma~\ref{lem:value-super-sol} below.

\end{proof}

\begin{lem}\label{lem:value-sub-sol}
For $i\in\{1,\ldots N\}$, the function $v_i$ satisfies
\[ 
 \lambda v_{i} (x)+\max\left\{H_{i}^{+}\left(x,\partial v_i (x) \right), -\lambda \cV|_\Gamma (x) \right\} \le 0,\quad \text{on }  \Gamma, 
 \]
in the viscosity sense.
\end{lem}

\begin{proof}
 Let $x\in\Gamma$. From Lemma~\ref{lem2.1} we have $v_i(x)\le \cV|_\Gamma (x)$. Hence, it suffices to prove that 
\[
\lambda v_i (x) + H^+_i \left( x, \partial v_i (x) \right) \le 0,
\]
in the viscosity sense.
Let $a_{i} \in A_i$ be such that $f_{i}\left(x,a_i\right) \cdot e_i >0$. 
By the Lipschitz continuity of $ f_i(\cdot,a_i) $, there exist $ r >0 $ such that $f_{i}\left(z,a_i\right) \cdot e_i >0$ for all $ z\in B(x,r) \cap (\cP_i \backslash \Gamma) $. 
Thus, there exists $ \tau >0$ such that for all $ z\in B(x,r) \cap (\cP_i \backslash \Gamma) $, there exists  $ (y_z, \alpha_z) \in \cT_z$ for which
\[ 
\alpha_z (t) =
\begin{cases}
a_i & 0\le t \le \tau,\\
\hat{\alpha} (t-\tau) & t \ge \tau,
\end{cases}
\]
where $ \hat{\alpha} $ is chosen arbitrarily. It follows that $y_{z}\left(t\right)\in \cP_{i}\backslash \Gamma $ for all $ t\le \tau $.
In other words, the trajectory $ y_z $ cannot approach $\Gamma$ since the speed pushes it away
from $\Gamma$, for $y_{z} (t)\in \cP_{i}\cap B\left(\Gamma,r\right)$.
Note that it is not sufficient to choose $a_{i}\in A_{i}$ such that $f_i\left(x,a_{i}\right) \cdot e_i=0$ since it may lead to $f\left(z,a_{i}\right)\cdot e_i <0$ for all $z\in \cP_{i}\backslash \Gamma $.
Next, since $y_{z}\left(t\right)\in \cP_{i}\backslash \Gamma $ for all $ t\le \tau $, we have
\[
v_{i}\left(z\right)  \le  J\left(z,\alpha_{z}\right)=\int_{0}^{\tau}\ell_{i}\left(y_{z}\left(s\right),a_{i}\right)e^{-\lambda s}ds+e^{-\lambda\tau}J\left(y_{z}\left(\tau\right),\hat{\alpha}\right).
\]
This inequality holds for any $\hat{\alpha}$, thus
\begin{equation}
v_{i}\left(z\right)\le\int_{0}^{\tau}\ell_{i}\left(y_{z}\left(s\right),a_{i}\right)e^{-\lambda s}ds+e^{-\lambda\tau}v_{i}\left(y_{z}\left(\tau\right)\right).\label{2ineq_sub_property 1}
\end{equation}
Furthermore, since $f_{i}\left(\cdot,a\right)$ is Lipschitz continuous by $\left[A1\right]$, for all $t\in\left[0,\tau\right]$,
\begin{eqnarray*}
\left|y_{z}\left(t\right)-y_{x}\left(t\right)\right| & = & \left|z+\int_{0}^{t}f_{i}\left(y_{z}\left(s\right),a_{i}\right)ds-x-\int_{0}^{t}f_{i}\left(y_{x}\left(s\right),a_{i}\right)ds\right|\\
 & \le & \left|z-x\right|+L\int_{0}^{t}\left|y_{z}\left(s\right)-y_{x}\left(s\right)\right|ds,
\end{eqnarray*}
and by Gr{\"o}nwall's inequality, 
\[
\left|y_{z}\left(t\right)-y_{x}\left(t\right)\right|\le\left|z-x\right|e^{Lt}, \quad \text{for all } t\in [0,\tau],
\]
yielding that $y_{z}(s)$ tends to $y_{x}(s)$ and $\int_0^\tau \ell_i(y_{z}(s),\alpha_z (s) ) ds$ tends to $\int_0^\tau \ell_i(y_{x}(s),\alpha_x (s) ) ds$
when $z$ tends to $x$. Hence, from~\eqref{2ineq_sub_property 1},
by letting $z \rightarrow x$, we obtain
\[
v_{i}\left(x\right)\le\int_{0}^{\tau}\ell_{i}\left(y_{x}\left(s\right),a_{i}\right)e^{-\lambda s}ds+e^{-\lambda\tau}v_{i}\left(y_{x}\left(\tau\right)\right).
\]
Let $\varphi$ be a function in $C^{1}\left(\cP_{i}\right)$
such that $0=v_{i}\left(x\right)-\varphi\left(x\right)=\max_{\cP_{i}}\left(v_{i}-\varphi\right)$.
This yields
\[
\dfrac{\varphi\left(x\right)-\varphi\left(y_{x}\left(\tau\right)\right)}{\tau}\le\dfrac{1}{\tau}\int_{0}^{\tau}\ell_{i}\left(y_{x}\left(s\right),a_{i}\right)e^{-\lambda s}ds+\dfrac{(e^{-\lambda\tau}-1) v_{i}\left(y_{x}\left(\tau\right)\right)}{\tau}.
\]
By letting $\tau$ tend to $0$, we obtain that
$
- f_{i}\left(x,a_{i}\right) \cdot \partial \varphi (x) \le \ell_{i}\left(x,a_{i}\right) - \lambda v_{i}\left(x\right)
$.
Hence,
\[
\lambda v_{i}\left(x\right)+\sup_{a\in A_{i}:f_{i}\left(x,a\right) \cdot e_i >0}\left\{ - f_{i}\left(x,a_{i}\right) \cdot \partial \varphi (x) - \ell_{i}\left(x,a_{i}\right)\right\} \le0.
\]
 Finally, by $ [A] $, it is easy to check that
\[
\sup_{a\in A_{i}:f_{i}\left(x,a\right) \cdot e_i >0}\left\{ - f_{i}\left(x,a_{i}\right) \cdot \partial \varphi (x) - \ell_{i}\left(x,a_{i}\right)\right\} =\max_{a\in A_{i}:f_{i}\left(x,a\right) \cdot e_i \ge0}\left\{ - f_{i}\left(x,a_{i}\right) \cdot \partial \varphi (x) - \ell_{i}\left(x,a_{i}\right)\right\} .
\]
The proof is complete.
\end{proof}

Before we give a proof of the fact that $v_i$ is a viscosity supersolution of~\eqref{eq:HJG3}, we prove the  following useful lemma.
\begin{lem}\label{lem:remain}
Let $x\in \Gamma$ and assume that
\begin{equation}\label{520}
v_i(x) < \cV|_\Gamma (x).
\end{equation}
Then, there exist $\bar{\tau}>0$ and $r>0$  such that for any $z\in (\cP_i \backslash \Gamma) \cap B(x,r) $, any $\varepsilon$ sufficiently small and any $\varepsilon$-optimal control law $\alpha^{\varepsilon}_z$ for $z$,
\[
y_{z,\alpha^\varepsilon_z} (s) \in \cP_i \backslash \Gamma, \quad \text{for all } s\in [0,\bar{\tau}].
\]
\end{lem}

This lemma means that if~\eqref{520} holds, then any trajectories starting from  $z\in (\cP_i \backslash \Gamma) \cap B(x,\varepsilon)$ still remains on $\cP_i \backslash \Gamma$ for a fixed amount of time.
Hence, this lemma takes into account the situation that the trajectory does not leave $\cP_i \backslash \Gamma$.

\begin{proof}[Proof of Lemma~\ref{lem:remain}]
We proceed by contradiction. Suppose that there exist sequences of positive numbers $\{\varepsilon_n\}, \{\tau_n\}$ and $\{x_n\}\subset \cP_i \backslash \Gamma$ such that $\varepsilon\rightarrow 0^+$, $x_n\rightarrow x$, $\tau_n \rightarrow 0^+ $ and $(y_{x_n}, \alpha_n) \in \cT_{x_n} $ where $\alpha_n$, $\varepsilon_n$-optimal control law,  satisfies $y_{x_n} (\tau_n) \in \Gamma$.
This implies that
\[
v_i(x_n) + \varepsilon_n > J(x_n,\alpha_n) = \int_0^{\tau_n} \ell (y_{x_n}(s), \alpha_n (s)) e^{-\lambda s}ds + e^{-\lambda \tau_n} J(y_{x_n}(\tau_n),\alpha_n (\cdot+\tau_n)).
\]
Since $\ell$ is bounded by $M$ by $[A1]$, then $v_i(x_n)+\varepsilon_n \ge -\tau_n M + e^{- \lambda \tau_n}\cV|_\Gamma (y_{x_n}(\tau_n))$.
Take a limit at infinity as $n\rightarrow \infty$, we get
$v_i(x) \ge \cV|_\Gamma (x)$ which contradicts \eqref{520}.

\end{proof}

\begin{lem}\label{lem:value-super-sol}
The function $v_i$ is a viscosity supersolution of \eqref{eq:HJG3}.
\end{lem}

\begin{proof}
Let $x\in \Gamma$. 
From Lemma~\ref{lem2.1}, we have $ v_j (x) \le \cV|_\Gamma(x) $ for all $ j $.
This yields that if the inequality \eqref{520} does not hold then $ v_i (x) = \cV|_\Gamma (x) $ and therefore $ v_i $ satisfies
\[
\lambda v_{i} (x)+\max\left\{H_{i}^{+}\left(x,\partial v_i (x) \right),-\lambda \cV|_\Gamma (x) \right\}  \ge \lambda v_i (x) - \lambda \cV|_\Gamma (x) = 0.
\]
Hence, in the rest of the proof, we assume that the inequality \eqref{520} holds and we aim to prove that
\[
\lambda v_i (x) + H^+_i \left( x, \partial v_i (x) \right) \ge 0,
\]
in the viscosity sense.
Let $\varphi\in C^{1}\left(\cP_{i}\right)$ be such that
\begin{equation}
0=v_{i}\left(x\right)-\varphi\left(x\right)\le v_{i}\left(z\right)-\varphi\left(z\right),\quad\mbox{for all } z\in \cP_{i},\label{532}
\end{equation}
and $\left\{ x_{\varepsilon}\right\} \subset \cP_{i}\backslash \Gamma $
be any sequence such that $x_{\varepsilon}$ tends to $x$ when $\varepsilon$
tends to $0$. From the dynamic programming principle and Lemma~\ref{lem:remain},
there exists $\bar{\tau}$ such that for any $\varepsilon>0$, there
exists $\left(y_{\varepsilon},\alpha_{\varepsilon}\right):=\left(y_{x_{\varepsilon}},\alpha_{\varepsilon}\right)\in\mathcal{T}_{x_{\varepsilon}}$
such that $y_{\varepsilon}\left(\tau\right)\in \cP_{i}\backslash \Gamma $
for all $\tau\in\left[0,\bar{\tau}\right]$ and 
\[
v_{i}\left(x_{\varepsilon}\right)+\varepsilon\ge\int_{0}^{\tau}\ell_{i}\left(y_{\varepsilon}\left(s\right),\alpha_{\varepsilon}\left(s\right)\right)e^{-\lambda s}ds+e^{-\lambda\tau}v_{i}\left(y_{\varepsilon}\left(\tau\right)\right).
\]
Then, according to~\eqref{532},
\begin{equation}
v_{i}\left(x_{\varepsilon}\right)-v_{i}\left(x\right)+\varepsilon \ge  \int_{0}^{\tau}\ell_{i}\left(y_{\varepsilon}\left(s\right),\alpha_{\varepsilon}\left(s\right)\right)e^{-\lambda s}ds+e^{-\lambda\tau}\left[\varphi\left(y_{\varepsilon}\left(\tau\right)\right)-\varphi\left(x\right)\right] -v_{i}\left(x\right) (1-e^{-\lambda\tau} ).\label{535}
\end{equation}
Next, one has
\[
\left\{\begin{array}{lll}
\ds \int_{0}^{\tau}\ell_{i}\left(y_{\varepsilon}\left(s\right),\alpha_{\varepsilon}\left(s\right)\right)e^{-\lambda s}ds = \ds \int_{0}^{\tau}\ell_{i}\left(y_{\varepsilon}\left(s\right),\alpha_{\varepsilon}\left(s\right)\right)ds+o\left(\tau\right),\\
 \left[\varphi\left(y_{\varepsilon}\left(\tau\right)\right)-\varphi\left(x\right)\right]e^{-\lambda\tau} =\varphi\left(y_{\varepsilon}\left(\tau\right)\right)-\varphi\left(x\right)+\tau o_{\varepsilon}\left(1\right)+o\left(\tau\right),
\end{array}
\right.
\]
and 
\[
\left\{\begin{array}{lll}
v_{i}\left(x_{\varepsilon}\right)-v_{i}\left(x\right)=o_{\varepsilon}\left(1\right),\\
 v_{i}\left(x\right) (1-e^{-\lambda\tau})=o\left(\tau\right)+\tau\lambda v_{i}\left(x\right),
\end{array}
\right.
\]
where the notation $o_{\varepsilon}\left(1\right)$ is used for a quantity
which is independent on $\tau$ and tends to $0$ as
$\varepsilon$ tends to $0$. 
For a positive integer $k$, the
notation $o(\tau^{k})$
is used for a quantity that is independent on $\varepsilon$ and such that
$o(\tau^{k})/\tau^{k}\rightarrow 0$ as $\tau\rightarrow0$.
Finally, $\mathcal{O}(\tau^k)$ stands for a quantity independent on $\varepsilon$
 such that $\mathcal{O}(\tau^{k})/\tau^{k}$ 
 remains bounded as $\tau\rightarrow0$. 
From~\eqref{535}, we obtain that
\begin{equation}
\tau\lambda v_{i}\left(x\right)\ge\int_{0}^{\tau}\ell_{i}\left(y_{\varepsilon}\left(s\right),\alpha_{\varepsilon}\left(s\right)\right)ds+\varphi\left(y_{\varepsilon}\left(\tau\right)\right)-\varphi\left(x\right)+\tau o_{\varepsilon}\left(1\right)+o\left(\tau\right)+o_{\varepsilon}\left(1\right).\label{537}
\end{equation}
Since $y_{\varepsilon}\left(\tau\right)\in \cP_{i}$ for all $\varepsilon$, we have
\[
\varphi\left(y_{\varepsilon}\left(\tau\right)\right)-\varphi\left(x_{\varepsilon}\right)=\int_{0}^{\tau} \partial \varphi\left(y_{\varepsilon}\left(s\right)\right) \cdot \dot{y}_{\varepsilon}\left(s\right)ds=\int_{0}^{\tau} \partial \varphi\left(y_{\varepsilon}\left(s\right)\right) \cdot f_{i}\left(y_{\varepsilon}\left(s\right),\alpha_{\varepsilon}\left(s\right)\right)ds.
\]
Hence, from~\eqref{537}
\begin{equation}
\begin{array}{ccc}
\tau\lambda v_{i}\left(x\right)-{\displaystyle \int_{0}^{\tau}}\left[\ell_{i}\left(y_{\varepsilon}\left(s\right),\alpha_{\varepsilon}\left(s\right)\right)+ \partial \varphi \left(y_{\varepsilon}\left(s\right)\right) \cdot f_{i}\left(y_{\varepsilon}\left(s\right),\alpha_{\varepsilon}\left(s\right)\right)\right]ds  \ge  \tau o_{\varepsilon}\left(1\right)+o\left(\tau\right)+o_{\varepsilon}\left(1\right).\end{array}\label{538}
\end{equation}
Moreover,
$\varphi\left(x_{\varepsilon}\right)-\varphi\left(x\right)=o_{\varepsilon}\left(1\right)$
and that $\partial \varphi \left(y_{\varepsilon}\left(s\right)\right)=\partial \varphi\left(x\right)+o_{\varepsilon}\left(1\right)+\mathcal{O}\left(s\right)$.
Thus 
\begin{equation}
\ds \lambda v_{i}\left(x\right)-\dfrac{1}{\tau}\int_{0}^{\tau}\left[\ell_{i}\left(y_{\varepsilon}\left(s\right),\alpha_{\varepsilon}\left(s\right)\right)+ \partial \varphi\left(x\right) \cdot f_{i}\left(y_{\varepsilon}\left(s\right),\alpha_{\varepsilon}\left(s\right)\right)\right]ds  \ge  o_{\varepsilon}\left(1\right)+\dfrac{o\left(\tau\right)}{\tau}+\dfrac{o_{\varepsilon}\left(1\right)}{\tau}.\label{539}
\end{equation}
Let $\varepsilon_{n}\rightarrow0$ as $n\rightarrow\infty$ and $\tau_{m}\rightarrow0$
as $m\rightarrow\infty$  such that 
\[
{\displaystyle \left(a_{mn},b_{mn}\right):=\left(\dfrac{1}{\tau_{m}}\int_{0}^{\tau_{m}}f_{i}\left(y_{\varepsilon_{n}}\left(s\right),\alpha_{\varepsilon_{n}}\left(s\right)\right)ds,\dfrac{1}{\tau_{m}}\int_{0}^{\tau_{m}}\ell_{i}\left(y_{\varepsilon_{n}}\left(s\right),\alpha_{\varepsilon_{n}}\left(s\right)\right)ds\right)}\longrightarrow\left(a,b\right)
\]
as $n,m\rightarrow\infty$. In view of  $\left[A1\right]$ and $\left[A2\right]$, we have
\begin{equation*}\begin{array}{lll}
\max\{ |f_i(y_{\varepsilon_n} (s) , \alpha_{\varepsilon_n} (s)) - f_i(x , \alpha_{\varepsilon_n} (s))|, |\ell_i(y_{\varepsilon_n} (s) , \alpha_{\varepsilon_n} (s)) - \ell_i(x , \alpha_{\varepsilon_n} (s))| \} \\ \qquad\qquad\qquad\qquad\qquad\qquad  \le  L | y_{\varepsilon_n} (s) - x| \le  L | y_{\varepsilon_n} (s) - x_{\varepsilon_n} | + L | x_{\varepsilon_n} - x| \\ 
 \qquad\qquad\qquad\qquad\qquad\qquad  \le LM \tau_m  + L |x_{\varepsilon_n} - x|.
\end{array}
\end{equation*}
Therefore,
\[
\begin{cases}
f_{i}\left(y_{\varepsilon_{n}}\left(s\right),\alpha_{\varepsilon_{n}}\left(s\right)\right) \cdot e_{i} = f_{i}\left(x,\alpha_{\varepsilon_{n}}\left(s\right)\right) \cdot e_{i}+o_{n}\left(1\right)+o_{m}\left(1\right),\\
\ell_{i}\left(y_{\varepsilon_{n}}\left(s\right),\alpha_{\varepsilon_{n}}\left(s\right)\right) =  \ell_{i}\left(x,\alpha_{\varepsilon_{n}}\left(s\right)\right) + o_{n}\left(1\right)+o_{m}\left(1\right).
\end{cases}
\]
Hence,
\begin{align*}
\left(a_{mn},b_{mn}\right) & ={\displaystyle \left(\dfrac{1}{\tau_{m}}\int_{0}^{\tau_{m}} f_{i}\left(x,\alpha_{\varepsilon_{n}}\left(s\right)\right)  ds ,\dfrac{1}{\tau_{m}}\int_{0}^{\tau_{m}}\ell_{i}\left(x,\alpha_{\varepsilon_{n}}\left(s\right)\right)ds\right)+o_{n}\left(1\right)+o_{m}\left(1\right)}\\
 & \in\text{FL}_{i}\left(x\right)+o_{n}\left(1\right)+o_{m}\left(1\right),
\end{align*}
since $\text{FL}_{i}\left(x\right)$ is closed and convex.
Let $n,m\rightarrow\infty$, then $\left(a,b\right)\in\text{FL}_{i}\left(x\right)$
and therefore there exists $\overline{a}\in A_{i}$ such that
\begin{align}
\lim_{m,n\rightarrow\infty}\left(\dfrac{1}{\tau_{m}}\int_{0}^{\tau_{m}}f_{i}\left(y_{\varepsilon_{n}}\left(s\right),\alpha_{\varepsilon_{n}}\left(s\right)\right) ds,\dfrac{1}{\tau_{m}}\int_{0}^{\tau_{m}}\ell_{i}\left(y_{\varepsilon_{n}}\left(s\right),\alpha_{\varepsilon_{n}}\left(s\right)\right)ds\right)=\left(f_{i}\left(x,\overline{a}\right),\ell_{i}\left(x,\overline{a}\right)\right).\label{540}
\end{align}
On the other hand, by Lemma~\ref{lem:remain}, $y_{\varepsilon_{n}}\left(s\right)\in \cP_{i}\backslash \Gamma $ for all $s\in\left[0,\tau_{m}\right]$. This yields 
\[
y_{\varepsilon_{n}}\left(\tau_{m}\right)=x_{\varepsilon_{n}}+\int_{0}^{\tau_{m}}f_{i}\left(y_{\varepsilon_{n}}\left(s\right),\alpha_{\varepsilon_{n}}\left(s\right)\right)ds.
\]
Since $y_{\varepsilon_{n}}\left(\tau_{m}\right) \cdot e_i >0$, then
\[
\dfrac{1}{\tau_{m}}\int_{0}^{\tau_{m}} f_{i}\left(y_{\varepsilon_{n}}\left(s\right),\alpha_{\varepsilon_{n}}\left(s\right)\right)  ds \cdot e_i\ge -\dfrac{x_{\varepsilon_{n}} \cdot e_i}{\tau_{m}} \ge  -\dfrac{\left|x_{\varepsilon_{n}}\right|}{\tau_{m}}.
\]
Let $\varepsilon_{n} \rightarrow 0$ then let $\tau_{m} \rightarrow 0$,  to obtain  $f_{i}\left(x,\overline{a}\right) \cdot e_i \ge0$, thus $\overline{a}\in A_{i}^{+}(x)$.
Hence, from~\eqref{539} and~\eqref{540}, replacing
$\varepsilon$ by $\varepsilon_{n}$ and $\tau$ by $\tau_{m}$, let
$\varepsilon_{n} \rightarrow 0 $, then let $\tau_{m} \rightarrow 0$,
we finally obtain
\[
\lambda v_{i}\left(x\right)+\max_{a\in A_{i}^{+}(x)}\left\{ -f_{i}\left(x,a\right) \cdot \partial \varphi \left(x\right)-\ell_{i}\left(x,a\right)\right\} \ge\lambda v_{i}\left(x\right)+\left[-f_{i}\left(x,\overline{a}\right) \cdot \partial \varphi \left(x\right)-\ell_{i}\left(x,\overline{a}\right)\right] \ge 0.
\]
\end{proof}

%%%%%%%%%%%%%%%%%%%%%%%%%%%%%%%%%%%%%%%%%
%%%%%%%%%%%%%%%%%%%%%%%%%%%%%%%%%%%%%%%%%
\subsection{A Comparison Principle and Uniqueness }\label{sec6}
In this section we establish a comparison principle for the Hamilton-Jacobi system~\eqref{eq:HJmain}.
From the comparison principle, it easily follows that $V:=(v_1,\ldots,v_N,\cV|_\Gamma)$ is the unique viscosity solution of~\eqref{eq:HJmain}.
\begin{thm}[Comparison Principle]\label{thm3}
Under Assumptions $[A]$ and $[A3]$, let $U$ and $W$ be respectively bounded continuous viscosity sub and supersolution of~\eqref{eq:HJmain}. Then $U\le W$ componentwise.
\end{thm}
We are going to give two proofs of Theorem~\ref{thm3}. 
The first one, given below, is inspired by Lions \& Souganidis \cite{LS2016,LS2017} by using arguments from the theory of PDE. The second one (displayed in the appendix)  is inspired by the works of Achdou, Oudet \& Tchou \cite{AOT2015} and Barles, Briani \& Chasseigne \cite{BBC2013,BBC2014} by using arguments from the theory of optimal control and PDE techniques. Both proofs make use of the following important properties of viscosity subsolutions displayed in the next lemma.

\begin{lem}\label{lem:continuity-sub}
Under  Assumptions $[A]$ and $[A3]$, let  $ U =  (u_1,\ldots,u_N,u_\Gamma)$ be a bounded continuous viscosity subsolution of~\eqref{eq:HJmain}, for any $i\in \{1,\ldots,N\}$ and $x\in\Gamma$,
the function $u_i$ is Lipschitz continuous in $B(\Gamma,r) \cap \cP_i$.
Therefore, there exists a test-function $\varphi_i \in C^1 (\cP_i)$ touching $u_i$ from above at $x$.
\end{lem}

\begin{proof}
	The proof of Lemma~\ref{lem:continuity-sub} is based on the fact that if $ U = (u_1,\ldots,u_N,u_\Gamma) $ is a viscosity subsolution of~\eqref{eq:HJmain}, then for any $ i\in \{1,\ldots,N\} $, $ u_i $ is a viscosity subsolution of
	\begin{equation}\label{001}
	\begin{cases}
	\lambda u_i (x) + H_i \left(x,\partial u_i (x) \right) &=0, \quad \text{if } x\in \cP_i \backslash \Gamma,\\
	\lambda u_i (x) + H^+_i \left(x,\partial u_i (x) \right) &=0, \quad \text{if } x\in \Gamma.\\
	\end{cases}
	\end{equation}
	Therefore, the proof is complete by applying the result in \cite[Section 3.2.3]{Oudet2016_these} (which is based on the proof of Ishii \cite{Ishii2013}).
 \end{proof}

\begin{proof}[A first proof of Theorem~\ref{thm3}] 
First of all, we claim that there exists a positive constant $ \bar{M} $ such that $ (\phi_1,\ldots\phi_N, \phi_\Gamma) $, where
$\phi_j: \cP_j \rightarrow \R$, $\phi_j(x)= -|x|^2 - \bar{M}$ and $\phi_\Gamma: \Gamma \rightarrow \R$, $\phi_\Gamma(z)= -|z|^2 - \bar{M}$, is a viscosity subsolution of~\eqref{eq:HJmain}.
Indeed, for any $ j \in \{1,\ldots,N\} $, since $ f_j $ and $ \ell_j $ are bounded by $ M $, one has
\[
H_j(x,p) = \max_{a\in A_j} \{-f_j(x,a) \cdot p - \ell_j (x,a)\} \le M (|p| + 1 ).
\]
Thus, using Cauchy-Schwarz inequality, there exists $ \bar{M} >0 $ such that
\[
\lambda (-|x|^2 - \bar{M}) + M (2|x| + 1 ) < 0, \quad \text{for all } x\in \cP_j.
\]
The claim for $ \phi_j $ is proved and we can prove similarly the one for $ \phi_\Gamma $.
Next, for $ 0<\mu<1 $, $ \mu $ close to $ 1 $, setting $ u_j^\mu  = \mu u_j + (1-\mu) \phi_j $ and $ u_\Gamma^\mu  = \mu u_\Gamma + (1-\mu) \phi_\Gamma $, then  $ (u_1^\mu, \ldots, u_N^\mu, u_\Gamma^\mu) $ is a viscosity subsolution of~\eqref{eq:HJmain}.  
Moreover, since $ u^\mu_j $ and $ u^\mu_\Gamma $ tend to $ -\infty $ as $ |x| $ and $ |z| $ tend to $ +\infty $ respectively, the functions $ u_j^\mu - w_j $ and $ u_\Gamma^\mu - w_\Gamma $ have maximum values $ M_j^\mu $ and $ M_\Gamma^\mu $ which are reached at some points $ \bar{x}_j $ and $ \bar{x}_\Gamma $ respectively. We argue by contradiction, through considering the two following cases: 

\begin{enumerate}
	\item [Case A] Assume that $ M_i^\mu > M^\mu_\Gamma $ and $ M_i^\mu>0 $. 
	Since $ (u_1^\mu, \ldots, u_N^\mu, u_\Gamma^\mu) $ is a viscosity subsolution of~\eqref{eq:HJmain}, by Lemma~\ref{lem:continuity-sub},
	there exists a positive number $L$ such that  $u^\mu_i$ is Lipschitz continuous with Lipschitz constant $L$ in $\cP_{i}\cap B(\bar{x}_i,r)$.
	We consider the function $ \Psi_{i,\varepsilon}:\cP_{i}\times \cP_{i} \rightarrow \R $ which is defined by
	\begin{equation}\label{Psi}
	\Psi_{i,\varepsilon}(x,y)  := u_i^\mu\left(x\right)-w_{i}\left(y\right)-\dfrac{1}{2\varepsilon} \left( ( -x^0+y^0+\delta\left(\varepsilon\right) )^{2}
	+ ( -x^i + y^i + \delta\left(\varepsilon\right) )^{2} \right)
	-|x-\bar{x}_i|^2,
	\end{equation}
	where $ \varepsilon >0 $, $\delta (\varepsilon ) :=\left(L+1\right)\varepsilon$ and $x=(x^i,x^0),y=(y^i,y^0)\in \R e_i \times \R e_0$. 
	It is clear that  $\Psi_{i,\varepsilon}$
	attains its maximum $M_{\varepsilon,\gamma}$ at $\left(x_{\varepsilon},y_{\varepsilon}\right)\in \cP_{i}\times \cP_{i}$.
	We claim that 
	\begin{equation}\label{601}
	\begin{cases}
	\ds M_{\varepsilon} \rightarrow \max_{\cP_i} \{u^\mu_i - w_i\} = u^\mu_i (\bar{x}_i) - w_i (\bar{x}_i) ,\\
	x_\varepsilon, y_\varepsilon \rightarrow \bar{x}_i \text{ and }
	\dfrac{\left(x_{\varepsilon}-y_{\varepsilon}\right)^2}{\varepsilon} \rightarrow 0,
	\end{cases}  \text{as } \varepsilon\rightarrow0.
	\end{equation}
	Indeed, we have
	\begin{eqnarray}\notag
	M_{\varepsilon} &= & u_{i}^\mu \left(x_{\varepsilon}\right)-w_{i}\left(y_{\varepsilon}\right)-\dfrac{\left(-x_{\varepsilon}^0 + y_{\varepsilon}^0 + \delta\left(\varepsilon\right)\right)^2 + \left(- x_{\varepsilon}^i + y_{\varepsilon}^i + \delta\left(\varepsilon\right)\right)^2}{2\varepsilon} - |x_\varepsilon - \bar{x}_i|^2\\
	&\ge & u_{i}^\mu\left(\bar{x}_i\right)-w_{i}\left(\bar{x}_i\right) -\dfrac{\left(L+1\right)^{2}}{2}\varepsilon.\label{610}
	\end{eqnarray}
	Since $M_i^\mu = u_{i}^\mu (\bar{x}_i)-v_{i} (\bar{x}_i )>0$, $ M_\varepsilon $ is
	positive when $\varepsilon$ is small enough.
	Furthermore, since $ w_i $ is bounded and $ u_i^\mu $ is bounded from above, we have $ |x_\varepsilon- \bar{x}_i|^2$ is bounded and $ x_\varepsilon - y_\varepsilon \rightarrow 0  $ as $ \varepsilon \rightarrow 0 $. Hence, after extraction of a subsequence,  $x_{\varepsilon},y_{\varepsilon}\rightarrow \bar{x} \in \cP_i$
	as $\varepsilon\rightarrow0$, for some  $\bar{x} \in \cP_{i}$. 
	%and $ x_\gamma \rightarrow \hat{x}\in \cP_i $ as $ \gamma \rightarrow 0 $.
	Thus, from~\eqref{610} we obtain
	\begin{equation*}
		M_i^\mu \ge  u_i (\bar{x}) - w_i (\bar{x}) - |\bar{x} - \bar{x}_i|^2
		\ge 
		\limsup_{\varepsilon \rightarrow 0} M_{\varepsilon} \ge \liminf_{\varepsilon \rightarrow 0} M_{\varepsilon} \ge M_i^\mu.
	\end{equation*}
	This implies that $ M_\varepsilon \rightarrow M_i^\mu$, $ (x_\varepsilon - y_\varepsilon)^2 / \varepsilon \rightarrow 0 $ as $ \varepsilon \rightarrow 0 $ and $ \bar{x} = \bar{x}_i $.
	The claim is proved.
	
	From now on in this proof we only consider the case when $\bar{x}_i \in \Gamma $, since otherwise, the proof follows by applying the classical theory (see \cite{BD1997, Barles1994}). 
	We claim that $x_{\varepsilon} \notin \Gamma$ for $ \varepsilon$ small enough.
	Indeed, assume by contradiction that $x_{\varepsilon} \in \Gamma$, i.e. $ x_{\varepsilon}^i = 0 $, we have
	\begin{enumerate}
		\item If $y_{\varepsilon} \notin \Gamma$, then let  $ z_{\varepsilon} = (y^i_{\varepsilon},x^0_{\varepsilon}) $, we have
		\begin{align*}
M_{\varepsilon} &= u_{i}^\mu (x_{\varepsilon} )-w_{i} (y_{\varepsilon} )-\dfrac{\left(-|x_{\varepsilon}^0|+|y_{\varepsilon}^0|+\delta\left(\varepsilon\right)\right)^2 + \left(|y_{\varepsilon}^i|+\delta\left(\varepsilon\right)\right)^2 }{2\varepsilon}
-|x_{\varepsilon} - \bar{x}_i|^2\\
&\ge  u_{i}^\mu (z_{\varepsilon} )-w_{i} (y_{\varepsilon})-\dfrac{\left(-|x_{\varepsilon}^0|+|y_{\varepsilon}^0|+\delta\left(\varepsilon\right)\right)^2 + \left(-|y_{\varepsilon}^i|+|y_{\varepsilon}^i|+\delta\left(\varepsilon\right)\right)^2 }{2\varepsilon} - |z_\varepsilon - \bar{x}_i|^2.
		\end{align*}
		Since $u_i^\mu$ is Lipschitz continuous in $B\left(\bar{x}_i,r\right)\cap \cP_{i}$,  we see that for $\varepsilon$ small enough
		\begin{align*}
L|y^i_{\varepsilon}|\ge u_{i}\left(x_{\varepsilon}\right)-u_{i}\left(z_{\varepsilon}\right) & \ge\dfrac{|y^i_{\varepsilon}|^{2}}{2\varepsilon}+\dfrac{|y^i_{\varepsilon}|\delta\left(\varepsilon\right)}{\varepsilon}   + |x_{\varepsilon} - \bar{x}_i|^2  - |z_\varepsilon - \bar{x}_i|^2\\ & \ge          \dfrac{|y_{\varepsilon}^i|\delta\left(\varepsilon\right)}{\varepsilon} - |y^i_{\varepsilon}| | x_\varepsilon+ z_\varepsilon - 2\bar{x}_i|.
		\end{align*}
		Therefore, if $y^i_{\varepsilon,\gamma}\not=0$, then $L\ge L + 1 - | x_\varepsilon+ z_\varepsilon - 2\bar{x}_i| $, which
		leads to a contradiction, since $ x_\varepsilon $, $ z_\varepsilon $ tend to $ \bar{x}_i $ as $ \varepsilon \rightarrow 0 $.
		\item If $y_{\varepsilon }\in \Gamma$, i.e. $y_{\varepsilon }^i = 0$, then let $ z_{\varepsilon } = (\varepsilon,x_{\varepsilon }^0) $, we have
		\begin{align*}
		M_{\varepsilon } &= u_{i}\left(x_{\varepsilon }\right)-w_{i}\left(y_{\varepsilon }\right)-\dfrac{\left(-|x_{\varepsilon }^0|+|y_{\varepsilon }^0|+\delta\left(\varepsilon\right)\right)^2 + \delta\left(\varepsilon\right)^2 }{2\varepsilon} - |x_\varepsilon - \bar{x}_i|^2 \\
		&\ge  u_{i}\left(z_{\varepsilon }\right) - w_{i}\left(y_{\varepsilon }\right) - \dfrac{\left(-|x_{\varepsilon }^0|+|y_{\varepsilon }^0|+\delta\left(\varepsilon\right)\right)^2+\left(-|\varepsilon|+\delta\left(\varepsilon\right)\right)^2  }{2\varepsilon} - |z_\varepsilon - \bar{x}_i|^2 .
		\end{align*}
		Since $u^\mu_i$ is Lipschitz continuous in $B\left(\bar{x}_i,r\right)\cap \cP_{i}$, we obtain
		\[
		L\varepsilon\ge u_i^\mu (x_{\varepsilon })- u^\mu_{i} (z_{\varepsilon }) 
		\ge - \dfrac{\varepsilon}{2} + \delta(\varepsilon)-  \varepsilon | x_\varepsilon+ z_\varepsilon - 2\bar{x}_i|.
		\]
		This implies that $L\ge L + 1/2 - | x_\varepsilon+ z_\varepsilon - 2\bar{x}_i|$,
		which yields a contradiction since  $ x_\varepsilon $, $ z_\varepsilon $ tend to $ \bar{x}_i $ as $ \varepsilon \rightarrow 0 $.     
	\end{enumerate}
	The second claim is proved. 
	We consider the following three possible cases
	\begin{enumerate}
		\item [Case A.1] There exists a subsequence of $ \{y_{\varepsilon } \} $ (still denoted by $ \{y_{\varepsilon } \} $) such that $ y_{\varepsilon } \in \Gamma  $  and $ w_i (y_{\varepsilon }) \ge w_\Gamma (y_{\varepsilon })$.
		Since $ x_\varepsilon,y_\varepsilon \rightarrow \bar{x}_i $ as $ \varepsilon \rightarrow 0 $ and $ u_i^\mu $ is continuous, for $ \varepsilon $ small enough, we have
		\begin{equation}\label{611}
		w_\Gamma (y_\varepsilon) \le w_i (y_\varepsilon) < u_i^\mu (y_\varepsilon)  = u_i^\mu (x_\varepsilon) +o_\varepsilon(1)   \le u_\Gamma^\mu(y_\varepsilon).
		\end{equation}
		Recall that the second inequality of~\eqref{611} holds since $ M_i^\mu >0 $ and the last inequality of~\eqref{611} holds since $ u^\mu_i $ and $ u^\mu_\Gamma $ satisfy
		\begin{equation}\label{613}
		\lambda u^\mu_{i} (x)+\max\left\{H_{i}^{+}\left(x,\partial u^\mu_i (x) \right), -\lambda u^\mu_\Gamma (x) \right\} \le 0, 
		\end{equation}
		in the viscosity sense.
		From \eqref{611},
		\[
		u_i^\mu (x_\varepsilon)  - w_i (y_\varepsilon) + o_\varepsilon (1) \le u_\Gamma^\mu (y_\varepsilon) - w_\Gamma (y_\varepsilon) \le M_\Gamma^\mu.  
		\]
		Let $ \varepsilon \rightarrow 0 $, thanks to~\eqref{601}, $  \sup\{u_i^\mu-w_i\} = M_i^\mu \le M_\Gamma^\mu $, which leads us to a contradiction since we assume that $ M_i^\mu > M_\Gamma^\mu $ in Case A.
		
		\item [Case A.2] There exists a subsequence of  $ \{y_{\varepsilon } \} $ such that $ y_{\varepsilon } \in \Gamma  $  and
		\[ 	\lambda w_{i}\left(y_{\varepsilon }\right)+H^+_{i}\left(y_{\varepsilon },\dfrac{-x_{\varepsilon }+y_{\varepsilon }+\delta (\varepsilon)}{\varepsilon} \right)  \ge0.  \]
		On the one hand, since $ H_i^+(x,p)\le H_i (x,p) $ for all $ x\in \Gamma $ and $ p\in \R e_i \times \R e_0 $, we have
		\begin{equation}\label{602}
		\lambda w_{i}\left(y_{\varepsilon }\right)+H_{i}\left(y_{\varepsilon },\dfrac{-x_{\varepsilon }+y_{\varepsilon }+\delta\left(\varepsilon\right)}{\varepsilon } \right)  \ge0.
		\end{equation}
		On the other hand, we have a viscosity inequality for $ u^\mu_i $ at $ x_{\varepsilon} \in \cP_i \backslash \Gamma $:
		\begin{equation}\label{603}
		\lambda u_{i}^\mu \left(x_{\varepsilon }\right)+H_{i}\left(x_{\varepsilon },\dfrac{-x_{\varepsilon }+y_{\varepsilon } + \delta\left(\varepsilon \right)}{\varepsilon } + 2 (x_{\varepsilon } - \bar{x}_i)\right)  \le0.
		\end{equation}
		Subtracting \eqref{602} from \eqref{603}, we obtain
		\begin{align*}
		& \lambda (u^\mu_{i}\left(x_{\varepsilon }\right)-w_{i}\left(y_{\varepsilon }\right)) \\
		& \le  H_{i}\left(y_{\varepsilon },\dfrac{-x_{\varepsilon }+y_{\varepsilon }+\delta\left(\varepsilon\right)}{\varepsilon} \right)-H_{i}\left(x_{\varepsilon },\dfrac{-x_{\varepsilon }+y_{\varepsilon }+\delta\left(\varepsilon  \right)}{\varepsilon} + 2 (x_\varepsilon - \bar{x}_i) \right).
		\end{align*}
		In view of \textsc{[A1]} and \textsc{[A2]}, there exists $C_{i}>0$
		such that for any $x,y\in \cP_{i}$ and $p,q\in\mathbb{R}$
		\begin{align*}
		\left|H_{i}\left(x,p\right)-H_{i}\left(y,q\right)\right| & \le\left|H_{i}\left(x,p\right)-H_{i}\left(y,p\right)\right|+\left|H_{i}\left(y,p\right)-H_{i}\left(y,q\right)\right|\\
		& \le C_i \left|x-y\right|\left(1+\left|p\right|\right)+ C_i \left|p-q\right|.
		\end{align*}
		This, in turn, yields 
		\begin{equation*}
		\lambda (u^\mu_{i}\left(x_{\varepsilon  }\right)-w_{i}\left(y_{\varepsilon     }\right))   \le C_i \left[\left|x_{\varepsilon    }-y_{\varepsilon     }\right| (1 + \dfrac{\left|-x_{\varepsilon  } + y_{\varepsilon     }+\delta (\varepsilon)   \right|}{\varepsilon  }  + 2 |x_\varepsilon - \bar{x}_i| )  + 2 |x_\varepsilon - \bar{x}_i |\right].
		\end{equation*}
		Letting $\varepsilon$ to $\infty$ and applying~\eqref{601}, we
		obtain that $\max_{\cP_i} \{u^\mu_{i}-w_{i}\} = M_i^\mu \le 0$, which leads to a contradiction.
		\item [Case A.3] There exists a subsequence of  $ \{y_\varepsilon \} $ such that $ y_{\varepsilon} \in \cP_i \backslash \Gamma $. Since the inequalities \eqref{602} and \eqref{603} still hold, a contradiction is obtained by using the similar arguments as in the previous case.
	\end{enumerate}
	
	\item [Case B] Assume that 
	\begin{equation}\label{Case B}
	M_\Gamma^\mu \ge \max_j \{ M_j^\mu\} \text{ and } M_\Gamma^\mu>0 .
	\end{equation}
	We consider the function
	\begin{align*}
	\Phi_{\varepsilon}:\Gamma \times \Gamma  & \longrightarrow \R\\
	\left(\zeta,\xi\right) & \longrightarrow u_\Gamma^\mu  (\zeta) - w_\Gamma  (\xi) - \dfrac{|\zeta - \xi |^2}{2\varepsilon} - |\zeta - \bar{x}_\Gamma|^2.
	\end{align*}
	By classical arguments, $\Phi_{\varepsilon}$
	attains its maximum $K_{\varepsilon}$ at $\left(\zeta_{\varepsilon}, \xi_{\varepsilon}\right)\in \Gamma \times \Gamma $ and
	\begin{equation}\label{612}
	\begin{cases}
	\ds K_{\varepsilon} \rightarrow M_\Gamma^\mu = \max_\Gamma \{u^\mu_\Gamma - w_\Gamma\} = u^\mu_\Gamma (\bar{x}_\Gamma) - w_\Gamma (\bar{x}_\Gamma) ,\\
	\zeta_\varepsilon, \xi_\varepsilon \rightarrow \bar{x}_\Gamma \text{ and }
	\dfrac{\left(\zeta_{\varepsilon} - \xi_{\varepsilon}\right)^2}{\varepsilon} \rightarrow 0,
	\end{cases}  \text{as } \varepsilon\rightarrow0.
	\end{equation}
	We consider the following two cases:
	\begin{enumerate}
		\item [Case B.1] There exists a subsequence $ \{\xi_\varepsilon\} $ (still denoted by $ \{\xi_\varepsilon \} $) such that
		\[
		\lambda w_\Gamma (\xi_\varepsilon) + H_\Gamma \left(\xi_\varepsilon, \dfrac{\zeta_\varepsilon - \xi_\varepsilon}{2 \varepsilon} \right) \ge 0.
		\]
		We also have a viscosity inequality for $ u_\Gamma^\mu $ at $ \zeta_\varepsilon $
		\[ 
		\lambda u_\Gamma^\mu (\zeta_\varepsilon) + H_\Gamma \left(\zeta_\varepsilon, \dfrac{\zeta_\varepsilon - \xi_\varepsilon}{2 \varepsilon} + 2(\zeta_\varepsilon - \bar{x}_\Gamma) \right) \le 0.
		\]
		By applying the classical arguments, one has
		$ (\zeta_\varepsilon - \xi_\varepsilon)^2 / \varepsilon \rightarrow 0 $ and $ \zeta_\varepsilon, \xi_\varepsilon \rightarrow \bar{x}_\Gamma$ as $ \varepsilon \rightarrow 0 $.  
		Subtracting the two above inequalities and sending $ \varepsilon $ to $ 0 $, we obtain that $  u_\Gamma^\mu (\bar{x}_\Gamma) - w_\Gamma (\bar{x}_\Gamma) = M_\Gamma^\mu \le 0$, which contradicts~\eqref{Case B}.
		
		\item [Case B.2] There exist a subsequence $ \{\xi_\varepsilon\} $ and $ k\in \{1,\ldots,N\} $ such that
		\begin{equation}\label{CaseB-3}
		w_\Gamma (\xi_\varepsilon) \ge \min_j \{ w_j(\xi_\varepsilon) + c_j (\xi_\varepsilon)   \} = w_k(\xi_\varepsilon) + c_k (\xi_\varepsilon).
		\end{equation}
		Since $ \zeta_\varepsilon,\xi_\varepsilon \rightarrow \bar{x}_\Gamma $ as $ \varepsilon \rightarrow 0 $ (by \eqref{612}) and $ u_\Gamma^\mu $ is continuous, for $ \varepsilon $ small enough, we have
		\[ 
		w_k (\xi_\varepsilon) + c_k (\xi_\varepsilon) \le w_\Gamma (\xi_\varepsilon)  < u_\Gamma^\mu (\xi_\varepsilon) = u_\Gamma^\mu (\zeta_\varepsilon) + o_\varepsilon(1) \le u_k^\mu(\xi_\varepsilon) + c_k (\xi_\varepsilon).
		\]
		Recall that the second inequality holds since $ M_\Gamma^\mu >0 $ and the last one comes from the fact that $ u_\Gamma^\mu $ satisfies
		\[
		\lambda u_{\Gamma}^\mu (x) + \max \left\{ -\lambda \min_{i=1,\ldots,N} \{u^\mu_i (x) + c_i(x) \} , H_\Gamma \left( x,\dfrac{\partial u_\Gamma^\mu}{\partial e_0}  (x) \right) \right\} \ge 0,
		\]
		in the viscosity sense. 
		This implies that  
		\begin{equation}\label{CaseB-2}
		u_\Gamma^\mu (\zeta_\varepsilon) - w_\Gamma (\xi_\varepsilon) + o_\varepsilon (1) \le u_k^\mu (\xi_\varepsilon)  - w_k (\xi_\varepsilon)  \le M_k^\mu.  
		\end{equation}
		Letting $ \varepsilon \rightarrow 0 $ and applying~\eqref{612}, we get $ M_\Gamma^\mu \le u_k^\mu(\bar{x}_\Gamma) - w_k (\bar{x}_\Gamma) \le M_k^\mu $. 
		This implies from~\eqref{Case B} that
		\begin{equation}\label{CaseB-1}
		M_\Gamma^\mu = M_k^\mu =u_k^\mu (\bar{x}_\Gamma) -w_k (\bar{x}_\Gamma)>0.
		\end{equation}
		Now we apply similar arguments as in the proof of Case A by considering the function $ \Psi_{k,\varepsilon}:\cP_{k}\times \cP_{k} \rightarrow \R $ which is defined in~\eqref{Psi} with the index $ i $ replacing the index $ k $.
		Remark that we may assume $ \bar{x}_k = \bar{x}_\Gamma $ since $ \bar{x}_\Gamma $ is a maximum point of $ u_k^\mu - w_k $ by~\eqref{CaseB-1}.
		It follows that the three cases B.2.1, B.2.2 and B.2.3 which are similar to A.1, A.2 and A.3, respectively.
		If Case B.2.1 occurs, i.e. there exists a subsequence of $ \{y_\varepsilon\} $ (still denoted by $ \{y_\varepsilon\} $) such that $ w_k(y_\varepsilon) \ge w_\Gamma (y_\varepsilon)$. 
		Letting $ \varepsilon \rightarrow 0 $, one gets $ w_k (\bar{x}_\Gamma) \ge w_\Gamma (\bar{x}_\Gamma)$. In the other hands, letting $ \varepsilon \rightarrow 0 $ in~\eqref{CaseB-3}, one also gets $ w_\Gamma (\bar{x}_\Gamma) \ge w_k(\bar{x}_\Gamma) + c_k (\bar{x}_\Gamma) > w_k(\bar{x}_\Gamma)$ which leads to a contradiction.
		Finally, the two last cases are proved by using the same arguments as in the proofs of Case A.2 and Case A.3.
%		Finally, 
%		
%		repeat the proof of Case A with the index $ i $ replacing the index $ k $ and the assumption '$ M_i^\mu > M_\Gamma^\mu $' replacing~\eqref{CaseB-1}.
%		Note that we may assume $ \bar{x}_k = \bar{x}_\Gamma $. We claim that Case A.1 does not hold. Indeed, from~\eqref{CaseB-3},~\eqref{CaseB-2} and~\eqref{CaseB-1}, $ w_\Gamma (\bar{x}_k) \ge w_k(\bar{x}_k) + c_k (\bar{x}_k) > w_k(\bar{x}_k)$. This implies that $ w_\Gamma(y_\varepsilon) > w_k(y_\varepsilon)$ and thus Case A.1 does not occur. Finally, it is easy to check that if either Case A.2 or Case A.3 occur, we arrive to a contradiction.
	\end{enumerate}
%	\item [Case C] Assume that
%	\begin{equation}\label{caseC}
%	M_\Gamma^\mu = \max_j \{ M_j^\mu\} = M_k^\mu >0 \text{ for some } k .
%	\end{equation}
\end{enumerate}

Finally, we get $ M_\Gamma^\mu \le 0 $ and $ M_i^\mu \le 0$ for all $ i\in \{1,\ldots,N\} $ and $ \mu \in (0,1) $, $ \mu $ close to $ 1 $. We conclude by letting $ \mu  $ tend to $ 1 $.
\end{proof}

\begin{cor}
	The value function $\cV$ has the property that the vector $V=(v_1,\ldots,v_N, \cV|_\Gamma) $ is the unique viscosity solution of~\eqref{eq:HJmain}. 
\end{cor}

%%%%%%%%%%%%%%%%%%%%%%%%%%%%%%%%%%%%%%%%%
%%%%%%%%%%%%%%%%%%%%%%%%%%%%%%%%%%%%%%%%%

\section{Hamilton-Jacobi system under a moderate controllability condition near the interface}\label{sec8}
In this section we derive the Hamilton-Jacobi system (HJ) associated with the above optimal control problem and prove that the value function given by \eqref{eq:value-f} is the unique viscosity solution of that (HJ) system, under the condition $[\tilde{A}3]$ below, which is weaker that the strong controllability condition $[A3]$ used above. 
\begin{description}
	\item {$ [\tilde{A}3]$ (\textbf{Moderate controllability})} There exist positive numbers $\delta$ and $ R$ such that
	\begin{itemize}
		\item  for any $ i=1,\ldots,N $ and for $ x\in \Gamma $
		\begin{equation}\label{eq:750}
		[-\delta, \delta] \subset \{ f_i(x,a) \cdot e_i : a\in A_i \}.
		\end{equation}

		\item for any $ x\in \Gamma $, there exists $ j\in \{1,\ldots,N\} $ such that
		\begin{equation}\label{eq:751}
		 [-\delta, \delta] \subset \{ f_j(x,a) \cdot e_0 : a\in A_j^{\Gamma} \}. 
		\end{equation}
	\end{itemize}
\end{description}

\begin{rem}
$
\quad
$
\begin{enumerate}
\item $[\tilde{A}3] $ allows us to construct an admissible control law and a corresponding trajectory that goes from one point on $ B(\Gamma,R)\cap \cP_i $  to another one on $ \Gamma $, see Lemma~\ref{lem:control} below. 

\item Assumption $ [\tilde{A}3] $ is rather stronger than the assumption~\eqref{eq:750} called "normal controllability" (and denoted $ [\tilde{H}3] $) which is used in~\cite{Oudet2016_these} related to the case without entry costs.
 With only $ [\tilde{H}3] $ and the effect of the entry costs, we could not 
 %%use optimal control arguments as the ones used in the second proof of Theorem~\ref{thm3} (see the appendix) 
prove that our value function is continuous on $ \Gamma $ and establish a corresponding Hamilton-Jacobi system. 
The moderate assumption $[\tilde{A}3]$ allows us to overcome the difficulties induced by the entry costs.
%%priori discontinuity.
\end{enumerate}
\end{rem}

\begin{rem}\label{rem:R}
Since for any $ i \in \{1,\ldots,N\} $, $f_i$ is Lipschitz continuous with respect to the state variable (from Assumption $ [A1] $), it is easy to check that there exists a positive number $ R $ such that for all $ i\in \{1,\ldots,N\} $ and $ x\in B(\Gamma,R) \cap \cP_i  $
\[ 
[-\dfrac{\delta}{2},\dfrac{\delta}{2}] \subset \{ f_i(x,a) \cdot e_i : a\in A_i \}.
\]
\end{rem}

Under Assumptions~$ [A]$ and $ [\tilde{A}3] $, following the arguments used in \cite[Theorem 3.3.1]{Oudet2016_these}, it holds  that for any $ x\in \cS $, the set $ Y_x $ is not empty. Hence, we can define the set of admissible controlled trajectories starting from the initial datum $ x $ to be
\[
\cT_x=\left\{ \left(y_{x},\alpha\right)\in L_{loc}^{\infty}\left(\mathbb{R}^{+};\mathcal{M}\right):y_{x}\in Lip\left(\mathbb{R}^{+};\cS\right)\mbox{ and }y_{x}\left(t\right)=
x+\int_{0}^{t}f\left(y_{x}\left(s\right),\alpha\left(s\right)\right)ds\right\} .
\]
The cost functional $ \cJ $ associated to the trajectory $ (y_{x} , \alpha) \in \cT_x$ is defined by
\[
\cJ(x,\alpha)=\int_0^\infty \ell(y_{x} (t),\alpha (t)) e^{-\lambda t} dt + \sum^N_{i=1} \sum_{k\in K_i} c_i (x_{ik}) e^{-\lambda t_{ik}}, 
\]
and
\[ 
\ccV(x) = \inf_{(y_{x},\alpha) \in \cT_x} \cJ(x,\alpha).
\]

Compared to the proof of Lemma \ref{lem:continuity}, we cannot use the classical control theory arguments to prove that the value function $ \ccV $ is continuous on $ \cP_i \backslash \Gamma$ for all $ i\in \{1,\ldots,N\} $.
The main problem is that, unlike under Assumption $[A3]$,  with the new assumption $ [\tilde{A}3]$, for $ x,z $ close to $ \Gamma $, there is possibly no admissible trajectory $ y_{x} \in \cT_x $ from $ x $ to $ z$. 
We will later prove that $ \ccV $ is continuous on $\cP_i \backslash \Gamma$ for any $ i $ by using the comparison principle, but for the moment, $ \ccV|_{\cP_i \backslash \Gamma} $ is a priori a discontinuous function.
In order to deal with such a discontinuity, we use the following notions
\begin{defn}
	Let $ i\in \{1,\ldots,N\} $ and $ u_i: \cP_i \backslash \Gamma \rightarrow \R$ be a bounded function.
	\begin{itemize}
		\item The upper semi-continuous envelope of $ u $ is defined by
		\[
		 u^{\star}_i (x) = \limsup_{z\rightarrow x} u_i(z).
		\]
		\item The lower semi-continuous envelope of $ u_i $ is defined by
		\[
		u_{i\star} (x) = \liminf_{z\rightarrow x} u(z).
		\]
	\end{itemize}
	Notice that both functions $ u_i^\star $ and $ u_{i \star} $ are defined on $ \cP_i $ instead of $ \cP_i \backslash \Gamma $.
\end{defn}

\begin{lem}\label{lem:control}
	Under Assumptions $[A]$ and $[\tilde{A}3]$,
	\begin{itemize}
		\item[$(a)$] there exists a positive constant $C$ such that for all $ z\in \Gamma $ and $ x \in \cS \cap B(\Gamma,R)$,  there exist $ (y_x,\alpha_{x,z}) \in \cT_{x}$ and $ \tau_{x,z} $ such that $ z= y_x (\tau_{x,z}) $ and $ \tau_{x,z} \le C |x-z|  $,
		\item [$(b)$] $ \ccV|_\Gamma $ is Lipschitz continuous on $ \Gamma $.
	\end{itemize}
\end{lem}
\begin{proof} We first note that by applying the classical arguments, (b) is a direct consequence of (a) (see \cite{BD1997}).
	To prove (a), we consider the two following cases.
	\begin{enumerate}
		\item  [Case 1]$ x, z\in \Gamma$. From~\eqref{eq:751}, we can find  $ (y_x,\alpha_{x,z}) \in \cT_{x}$ for which $\alpha_{x,z}$ satisfies
		\[
		f\left(x+\dfrac{z-x}{|z-x|} \delta t, \alpha_{x,z} (t) \right)= \delta \dfrac{z-x}{|z-x|},\quad \text{for all } t\le \dfrac{|z-x|}{\delta}. 
		\]
		It is easy to check that $ y_x (|z-x|/\delta) =z $, i.e. $ \tau_{x,z} = |z-x|/\delta$. 
		
		\item  [Case 2]$ x = (x^i,x^0) \in \cP_i \backslash \Gamma $ and $ z = (0,z^0)\in \Gamma $. One the one hand, from~\eqref{eq:750} and~\eqref{eq:751}, we can pick $ (y_x,\alpha_{x,z}) \in \cT_{x}$ for which $\alpha_{x,z}$ satisfies
		\[
		\begin{cases}
		f\left( y_x(t) , \alpha_{x,z} (t) \right) \cdot e_i = - \dfrac{\delta}{2}, &  t\le \dfrac{2 x^i}{\delta},\\
		f\left( y_x(t) , \alpha_{x,z} (t) \right) = \delta \dfrac{z- y_x(2 x^i / \delta) }{|z- y_x(2 x^i / \delta)|}, &  \dfrac{2 x^i}{\delta} \le t\le \dfrac{2 x^i}{\delta} +  \dfrac{|z- y_x(2 x^i / \delta) |}{ \delta}.
		\end{cases}	
		\]
		This simply means that $ 2x^i / \delta $ is the exit time of $ y_x $  from $ \cP_i \backslash \Gamma $ and $ \tau_{x,z} = \frac{2 x^i}{\delta} +  \frac{|z- y_x(2 x^i / \delta) |}{ \delta}  $.
		On the other hand, let $ \bar{x}= (0,x^0) \in \Gamma $, since $f$ is bounded by $M$, it holds that
		\[
		| z - y_x ( 2 x^i /\delta )   | \le |z- \bar{x}| + |\bar{x} - y_x ( 2 x^i /\delta ) | \le |z-x|+ \dfrac{ x^i(4M^2 \delta^2 -1)^{1/2}}{\delta}  .
		\]
		Since $ x^i \le |z-x| $, we finally obtain
		\[ 
		\tau_{x,z} = \dfrac{2 x^i}{\delta} +  \dfrac{|z- y_x(2 x^i / \delta) |}{ \delta} \le \dfrac{|z-x|}{\delta} \left( 3+ \dfrac{(4M^2 \delta^2 -1)^{1/2}}{\delta} \right).
		\]
		 
	\end{enumerate}
\end{proof}

\subsection{Value function on the interface}
\begin{thm}\label{thm5}
	Under Assumptions $[A]$ and $[\tilde{A}3]$, the restriction of the value function $ \ccV $ to $\Gamma$,  $\ccV|_\Gamma$, is a unique viscosity solution of the equation
	\begin{equation}\label{eq:HJG4}
	\lambda u_\Gamma (x)+ \max \left\{-\lambda \min \{v_{i} (x) +c_i (x) \}, H_\Gamma \left(x,\dfrac{\partial u_\Gamma}{\partial e_0} (x)\right)\right\} = 0, \quad x\in \Gamma.
	\end{equation}
	Here a function $ u_\Gamma : \Gamma \rightarrow \R $ is called a viscosity solution of~\eqref{eq:HJG4} if it satisfies
	\begin{align}
	\lambda u_\Gamma (x)+ \max \left\{-\lambda \min \{v_i^\star (x) +c_i (x) \}, H_\Gamma \left(x,\dfrac{\partial u_\Gamma}{\partial e_0} (x)\right)\right\} &\le 0, \quad \text{in the viscosity sense}\\
	\lambda u_\Gamma (x)+ \max \left\{-\lambda \min \{v_{i \star} (x) +c_i (x) \}, H_\Gamma \left(x,\dfrac{\partial u_\Gamma}{\partial e_0} (x)\right)\right\} &\ge 0, \quad \text{in the viscosity sense}
	\end{align}
\end{thm}
The proof of Theorem~\ref{thm5} is a consequence of the next three lemmas which are similar to  Lemma~\ref{lem2.2}, Lemma~\ref{lem2.1} and Lemma~\ref{lem2.3}, respectively.

\begin{lem}
	The restriction of the value function $\ccV$ on the interface $\Gamma$, $\ccV|_{\Gamma}$ satisfies
	\begin{equation*}
	\lambda \ccV|_\Gamma (x)+  H_\Gamma \left(x,\dfrac{\partial \ccV|_\Gamma}{\partial e_0} (x)\right) \le 0, \quad x\in \Gamma,
	\end{equation*}
	 in the viscosity sense.
\end{lem}
\begin{proof}
The proof of this lemma is similar to the proof of Lemma~\ref{lem2.2}.	
\end{proof}

\begin{lem}\label{lem:weak-compare}
	For all $x\in \Gamma$ 
	\[
	\max_{i=1,\ldots,N} \{ v_i^\star (x) \} \le \ccV|_\Gamma (x) \le \min_{i=1,\ldots,N} \{v_{i}^{\star} (x) +c_i (x) \}.
	\]
\end{lem}

\begin{proof}
	
	Let $ i\in \{1,\ldots,N\} $, it suffices to prove (a) $v_i^\star (x) \le \ccV|_\Gamma (x)$ and (b) $\ccV|_\Gamma (x) \le v_{i}^{\star} (x)  + c_i (x) $ below.
	\begin{enumerate}
		\item [(a)] 	There exists a sequence $ \{x_n\}_{n\in \N} \subset (\cP_i \backslash \Gamma)  \cap B(\Gamma,R) $ such that $ x_n \rightarrow x $ and $ v_i (x_n) \rightarrow v_i^\star (x) $ as $ n \rightarrow \infty $. From Lemma~\ref{lem:control}, there exists a positive constant $ C $ such that there exist $ (y_{x_n}, \alpha_n) \in \cT_{x_n}$ and $ \tau_n $ such that $ y_{x_n} (\tau_n) = x $ and $ \tau_n \le |x_n - x| $ for all $ n $. This implies
		\begin{equation*}
		v_i (x_n) = \ccV(x_n) \le  \int_0^{\tau_n} \ell (y_{x_n} (s), \alpha_n (s) ) e^{-\lambda s} ds + \ccV(x) e^{-\lambda \tau_n} \le MC |x_n - x| + \ccV(x) e^{-\lambda \tau_n}.
		\end{equation*}
		Letting $ n\rightarrow \infty $, one gets $ v_i^\star (x) \le \ccV(x) $.
		\item [(b)] 	
		From Assumption $[\tilde{A}3]$, it allows to pick $ (y_{n},\alpha_n) \in \cT_{x}$ and $ \tau_n \rightarrow 0 $  where $\alpha_n$ satisfies
		\[
		f(y_n(t), \alpha_n (t)) \cdot e_i = \dfrac{\delta}{2} >0, \quad \text{for all } t\le \tau_n. 
		\]
		Let $ x_n := y_n(\tau_n) $ then $ x_n\in \cP_i \backslash \Gamma $ and  $ x_n \rightarrow x $ as $ n \rightarrow \infty$. 
		This implies
		\begin{eqnarray}
		\ccV(x) &\le &  \int_0^{\tau_n} \ell (y_{n} (s), \alpha_n (s) ) e^{-\lambda s} ds + c_i(x) + \ccV (x_n) e^{-\lambda \tau_n} \nonumber\\  
		&\le& MC \tau_n + c_i (x)  + v_i (x_n) e^{-\lambda \tau_n}. \nonumber
		\end{eqnarray}
		Taking limsup of both sides as $ n\rightarrow \infty $, one gets $ \ccV (x) \le c_i (x) + v_i^\star (x) $.
	\end{enumerate}

\end{proof}

\begin{lem}
	Let $ x\in \Gamma $ and assume that $\ccV|_\Gamma (x) < \min \{v_{i \star} (x) +c_i(x) \}$, then $ \ccV|_\Gamma $ satisfies
	\[
	\lambda \ccV|_\Gamma (x) + H_\Gamma \left( x, \dfrac{\partial \ccV|_\Gamma}{\partial e_0} (x) \right) \ge 0.   
	\]
	in the viscosity sense.
\end{lem} 

\begin{proof}
	The proof of this lemma is similar to the proof of Lemma~\ref{lem2.3}.
\end{proof}

%%%%%%%%%%%%%%%%%%%%%%%%%
 %%%%%%%%%%%%%%%%%%%%%%%%%%

\subsection{Hamilton-Jacobi system and viscosity solution}
\begin{defn}
	A function $ \cU := ( u_1,\ldots,u_N, u_\Gamma ) $, where each $ u_i: \cP_i \rightarrow \R $ is a bounded function and $ u_\Gamma : \Gamma \rightarrow \R $ is a bounded continuous function, is called a discontinuous viscosity solution of~\eqref{eq:HJmain} if $ \cU^\star = (u_1^\star, \ldots u_N^\star, u_\Gamma) $ is a viscosity subsolution of~\eqref{eq:HJmain} and  $ \cU_\star = (u_{1\star}, \ldots u_{N\star}, u_\Gamma) $ is a viscosity supersolution of~\eqref{eq:HJmain}.
\end{defn}

\begin{thm}
	Under Assumptions $[A]$ and $[\tilde{A}3]$, $ V:=(v_1,\ldots,v_N, \ccV|_\Gamma) $ is a viscosity solution of~\eqref{eq:HJmain}. 
\end{thm}
The proof of viscosity subsolution follows from that of Lemma~\ref{lem:value-sub-sol} where instead of using an arbitrary sequence which tends to $ x \in \Gamma $ as $ n\rightarrow \infty $, we work with a sequence $ \{x_n\}_{n\in \N} \subset \cP_i \backslash \Gamma $ which satisfies
\[
\begin{cases}
x_n \rightarrow x,\\
v_i(x_n) \rightarrow v_i^\star (x),
\end{cases}
\text{as } n\rightarrow \infty.  
\]

The first difference between the proof of viscosity supersolution with the one of Lemma~\ref{lem:value-super-sol} is the following lemma, which has a similar proof as Lemma~\ref{lem:remain}.

\begin{lem}
	Let $x\in \Gamma$ and assume that
	\begin{equation}\label{753}
	v_{i \star}(x) < \ccV|_\Gamma (x).
	\end{equation}
	Then, for any sequence $ \{z_n\}_{n\in \N}\subset \cP_i \backslash \Gamma $ such that
	\begin{equation}\label{754}
	\begin{cases}
	z_n \rightarrow x,\\
	v_i(z_n) \rightarrow v_{i \star} (x),
	\end{cases} \text{as } n\rightarrow \infty,
	\end{equation}
	  there exist $\bar{\tau}>0$ and $r>0$  such that, for any $\varepsilon$ sufficiently small and any $\varepsilon$-optimal control law $\alpha^{\varepsilon}_{z_n}$ for $z_n$,
	\[
	y_{z_n,\alpha^\varepsilon_{z_n}} (s) \in \cP_i \backslash \Gamma, \quad \text{for all } s\in [0,\bar{\tau}].
	\]
\end{lem}

The last difference between the proof of viscosity supersolution with the one of Lemma~\ref{lem:value-super-sol} is that instead of using an arbitrary sequence which tends to $ x \in \Gamma $ as $ n\rightarrow \infty $, we work with a sequence $ \{z_n\}_{n\in \N} \subset \cP_i \backslash \Gamma $ satisfying~\eqref{754}.

%%%%%%%%%%%%%%%%%%%%%%%%%%%%%%%%%%%%%%%%%
%%%%%%%%%%%%%%%%%%%%%%%%%%%%%%%%%%%%%%%%%

\subsection{Comparison Principle and Uniqueness}\label{sec10}
In the proof of the Comparison Principle under the strong controllability assumption $[A3]$, the key argument is that the restrictions of the viscosity subsolutions of~\eqref{eq:HJmain} to $\cP_i \backslash \Gamma$ are Lipschitz continuous in a neighborhood of $ \Gamma $. This property is not obtained directly in the current framework, under the moderate controllability assumption $[\tilde{A}3]$. We rather proceed as follows. First, we regularize a restriction of viscosity subsolution to $ \cP_i \backslash \Gamma$ using  sup-convolution to obtain a family of  Lipschitz continuous functions. Then we use this family of regularized functions to prove a local comparison principle which we finally extend to a global comparison principle by applying similar arguments as  the ones in the proof of \cite[Theorem 3.3.4]{Oudet2016_these}.

We begin the first step with sup-convolution definition.

\begin{defn}\label{def: alpha}
	Let $ u_i : \cP_i \rightarrow \R$ be a bounded, USC function and $ \alpha,p $ be positive numbers. Then sup-convolution of $ u_i $ with respect to the $ x^0 $-variable is defined by
	\[ 
	u_i^\alpha (x):= \sup_{z^0\in \R} \left\{  u( z^0 e_0 + x^i e_i ) - \left( \dfrac{|z^0 - x^0|^2}{\alpha^2} + \alpha \right)^\frac{p}{2}   \right\} , \quad \text{if } x= x^0 e_0 + x^i e_i .
	\] 
\end{defn}

We borrow the following  useful lemmas from \cite[Lemma 3.3.7 to 3.3.10]{Oudet2016_these}.

\begin{lem}\label{lem:11}
	Let $ u_i : \cP_i \rightarrow \R$ be a bounded function and $ \alpha,p $ be positive numbers.  Then, for any $ x\in \cP_i $, the supremum which defines $ u_i^\alpha $ is achieved at a point $ z^0\in \R $ such that
	\begin{equation}\label{121}
	\left(  \dfrac{|z^0 - x^0|^2}{\alpha^2} +\alpha \right) ^\frac{p}{2} \le 2 || u_i ||_{L^\infty (\cP_i)} + \alpha^\frac{p}{2}.
	\end{equation}
\end{lem}

\begin{lem}\label{lem:12}
	Let $ u_i : \cP_i \rightarrow \R$ be a bounded function. Then for any $ \alpha,p>0 $, the sup-convolution $ u_\alpha $ is locally Lipschitz continuous with respect to $ x^0 $, i.e. for any compact subset $ K $ in $ \R^3 $, there exists $ C_K>0 $ such that for all $ x= x^0 e_0 + x^i e_i $, $ y= y^0 e_0 + x^i e_i \in K \cap \cP_i$
	\[ 
	|u^\alpha_i (x) - u^\alpha_i (y)| \le C_K |x^0 - y^0|.
	\]
\end{lem}

\begin{lem}\label{lem:Lips1}
	Under Assumptions $[A]$ and $[\tilde{A}3]$, let $ R>0 $ be a positive number as in Remark~\ref{rem:R}.
	Let $ u : \cP_i \rightarrow \R $ be a bounded, USC subsolution of~\eqref{eq:HJmain} and $ \alpha,p $ be some positive numbers.
	Then for all $ M>0 $, $ u_i^\alpha $ is Lipschitz continuous in $ B_M (\Gamma, R) \cap \cP_i$, where $ B_M (\Gamma,R):=\{x\in B(\Gamma,R) : |x^0|\le M  \} $.
\end{lem}

\begin{lem}\label{lem:Lip2}
	Under Assumptions $[A]$ and $[\tilde{A}3]$, let $ \bar{y} \in \Gamma $ and $ R>0 $ be as in Remark~\ref{rem:R}.
	We denote by $ Q_i $ the set $ \cP_i \cap B(\bar{y}, R) $. Let $ u_i : \cP_i \rightarrow \R$ be a bounded, USC viscosity subsolution of~\eqref{eq:HJmain} in $ Q_i $.  Then for all $ \alpha, p>0 $ sufficient small, if we set
	\begin{equation}\label{eq:Qalpha}
	Q_i^\alpha = \left\{ x\in Q_i : \text{dist} (x,\partial Q_i) > \alpha \left(4^{1/p} ||u_i||^{2/p}_{L^\infty (\cP_i)} - \alpha\right)^{1/2} \right\},
	\end{equation}
	the function $ u_i^\alpha $ defined in Definition~\ref{def: alpha} is Lipschitz continuous in $ Q_i^\alpha $.
	Moreover, there exists $ m: (0,\infty) \rightarrow (0,\infty) $ such that $ \lim_{\alpha \rightarrow 0} m(\alpha) =0 $ and $ (u_1^\alpha  - m(\alpha), \ldots, u_N^\alpha  - m(\alpha) , u_\Gamma)$ is a viscosity subsolution of~\eqref{eq:HJmain} in $ Q^\alpha:= \cup_i Q_i^\alpha $.
\end{lem}
Lemma~\ref{lem:11} and Lemma~\ref{lem:12} on sup-convolution are well known result,  but we give a short proof for completeness.
First of all, it is easy to see that the supremum in Definition~\ref{def: alpha} is achieved at some point $ z^0 \in \R $.  We have
\[ 
u_i^\alpha (x) =  u( z^0 e_0 + x^i e_i ) - \left( \dfrac{|z^0 - x^0|^2}{\alpha^2} + \alpha^2 \right)^\frac{p}{2} \ge u( x^0 e_0 + x^i e_i ) + \alpha^\frac{p}{2},
\]
and because of the boundedness of $ u $, we get~\eqref{121}.
Next, let $ K $ be a compact subset of $ \cP_i $, for $ x,y\in K $, from the definition of sup-convolution, we have

\[ u_i^\alpha (x) - u_i^\alpha (y) \le \sup_{z^0\in \R} \left\{ - \left( \dfrac{|z^0 - x^0|^2}{\alpha^2} + \alpha \right)^\frac{p}{2} + \left( \dfrac{|z^0 - y^0|^2}{\alpha^2} + \alpha \right)^\frac{p}{2}  \right\}.
 \]
By the mean-value theorem and the fact that $ |z^0 - x^0|/\alpha^2 $, $  |z^0 - y^0|/\alpha^2$ are bounded, there exists a constant $ C_K>0 $ such that $ 	|u^\alpha_i (x) - u^\alpha_i (y)| \le C_K |x^0 - y^0| $. The proof of Lemma~\ref{lem:12} is complete.

Finally, if $ u_i $ is a viscosity subsolution of~\eqref{eq:HJmain}, then it is also a viscosity subsolution of the following equation
\begin{equation}
\begin{cases}
\lambda u_{i} (x) +H_{i}\left(x, \partial u_i (x) \right) \le0 & \text{if } x\in \cP_{i}\backslash \Gamma,\\
\lambda u_{i} (x)+ H_{i}^{+}\left(x,\partial u_i (x) \right)  \le 0 & \text{if } x\in \Gamma.
\end{cases}\label{eq:HJG5}
\end{equation}
We may now apply the proofs of~\cite[Lemmas 3.3.9 and 3.3.10]{Oudet2016_these} to conclude Lemmas~\ref{lem:Lips1} and \ref{lem:Lip2}.

We are ready to prove a local comparison principle.

\begin{thm}
	Under Assumptions $[A]$ and $[\tilde{A}3]$, let $ \cU=(u_1,\ldots,u_N, u_\Gamma)$ be a bounded, USC viscosity subsolution of~\eqref{eq:HJmain} and $ \cW= (w_1,\ldots,w_N,w_\Gamma)$ be a bounded, LSC viscosity supersolution of~\eqref{eq:HJmain}.
	Let $ R>0 $ as in Remark~\ref{rem:R}, $ i\in \{1, \ldots,N \} $ and $ \bar{y} \in \Gamma $ be fixed.
	Then, if we set $ Q_i = B(\bar{y}, R) \cap \cP_i$, we have
	\begin{align}\label{eq:local compar}
	|| (u_i- w_i)_+||_{L^\infty (Q_i)} &\le || (u_i- w_i)_+||_{L^\infty (\partial Q_i)}, \\
	 || (u_\Gamma- w_\Gamma)_+||_{L^\infty (Q_i\cap \Gamma)} &\le || (u_\Gamma- w_\Gamma)_+||_{L^\infty (\partial (Q_i\cap \Gamma))}.\label{eq:local compar1}
	\end{align}
	
\end{thm}

\begin{proof}
	Take $ \alpha,p>0 $ sufficient small so that Lemma~\ref{lem:Lip2} can be applied. From Lemma~\ref{lem:Lip2}, we know the fact that $ u_i^\alpha $ is Lipschitz continuous in $ Q_\alpha $ with a Lipschitz constant $ L_\alpha $ and that there exists $ m : (0,\infty) \rightarrow (0,\infty) $ such that $ \lim_{\alpha \rightarrow 0} m (\alpha)=0   $ and $(u_1^\alpha - m,\ldots,u_N^\alpha - m, u_\Gamma)  $ is a viscosity subsolution of~\eqref{eq:HJmain}. 
	Set $ \tu_j:= u_j^\alpha - m $ for all $ j\in \{1,\ldots,N\} $. Let us prove that
	\begin{align}\label{eq:local compar alpha}
	|| (\tu_i - w_i)_+||_{L^\infty (Q_\alpha)} &\le || (\tu_i - w_i)_+||_{L^\infty (\partial Q_\alpha)},\quad \text{for all } i\in \{1,\ldots,N\}\\
	|| (u_\Gamma- w_\Gamma)_+||_{L^\infty (Q^\alpha\cap \Gamma)} &\le || (u_\Gamma- w_\Gamma)_+||_{L^\infty (\partial (Q^\alpha\cap \Gamma))}. \label{eq:local compar alpha1}
	\end{align}
	We argue by contradiction by considering the two following cases 
	\begin{enumerate}
		\item [Case A] If there exist $ i\in \{1,\ldots,N\} $ and $ \bar{x}_i \in Q_i^\alpha \backslash \partial Q_i^\alpha $ such that 
	\begin{equation*}
		M_i:= \max_{Q_i^\alpha} \{\tu_i - w_i\} =  \tu_i (\bar{x}_i)- w_i(\bar{x}_i) > \max \{M_\Gamma, 0\},
	\end{equation*}
	where $ M_\Gamma := \max_{Q_i^\alpha \cap \Gamma} \{u_\Gamma - w_\Gamma\} $.  Now we can apply the similar arguments as in Case A in the proof of Theorem~\ref{thm3} to obtain that either $ M_i \le 0 $ or $ M_i \le M_\Gamma $, which leads us to a contradiction.
	\item [Case B] If there exists $ \bar{x}_\Gamma \in (Q_i^\alpha \backslash \partial Q_i^\alpha) \cap \Gamma $ such that
	\[ 
		M_\Gamma =  \max_{Q_i^\alpha \cap \Gamma} \{u_\Gamma - w_\Gamma\} =  u_\Gamma (\bar{x}_\Gamma)- w_i(\bar{x}_\Gamma) \ge  \max_{j\in \{1,\ldots,N\}} M_j \text{ and } M_\Gamma > 0.
 	 \]
 	 Apply the similar arguments as in Case B in the proof of Theorem~\ref{thm3} to obtain that either $ M_\Gamma \le 0 $ (which leads to a contradiction) or there exists $ k\in \{1,\ldots,N\} $ such that 
 	  \begin{equation}\label{017}
 	 w_\Gamma (\bar{x}_\Gamma) \ge \min_{j} \{w_j(\bar{x}_\Gamma) +c_j (\bar{x}_\Gamma )\} =  w_k(\bar{x}_\Gamma) +c_k (\bar{x}_\Gamma ), 
  \end{equation} which leads to $  M_k = \tu_k (\bar{x}_\Gamma) - w_k (\bar{x}_\Gamma) = M_\Gamma >0 $.
  We now repeat Case A and note that thanks to~\eqref{017}, the case $ w_k (\bar{x}_\Gamma ) \ge \tu_k (\bar{x}_\Gamma)  $ does not occur, we finally obtain that $ M_k \le 0  $ which contradicts~\eqref{017}. 
	\end{enumerate}
We finish the proof of ~\eqref{eq:local compar alpha} and \eqref{eq:local compar alpha1}.

In order to prove~\eqref{eq:local compar} and \eqref{eq:local compar1}, we have to pass to the limit  as $ \alpha $ tends to $ 0 $ in~\eqref{eq:local compar alpha} and \eqref{eq:local compar alpha1}, respectively. Let $ \alpha_0>0 $ and $ y \in Q_{\alpha_0} $ be fixed.  For all $ 0<\alpha \le \alpha_0 $, one has
\begin{equation}\label{758}
	 (u_i^\alpha (y) -m(\alpha)  - w_i (y))_+ \le || (u_i^\alpha - m(\alpha) - w_i)_+||_{L^\infty (\partial Q_\alpha)}.
\end{equation}
We claim that $ \limsup_{\alpha\rightarrow0}  || (u_i^\alpha - m(\alpha) - w_i)_+||_{L^\infty (\partial Q_\alpha)} \le || (u_i^\alpha - m(\alpha) - w_i)_+||_{L^\infty (\partial Q)}$.
Indeed, there exists $ x_\alpha \in \partial Q_\alpha $ such that
\[ 
(u_i^\alpha (x_\alpha) -m(\alpha)  - w_i (x_\alpha))_+ = || (u_i^\alpha - m(\alpha) - w_i)_+||_{L^\infty (\partial Q_\alpha)}.
 \]
Thus, for any subsequence such that $  || (u_i^\alpha - m(\alpha) - w_i)_+||_{L^\infty (\partial Q_\alpha)} $ converges to a limit $ \ell $ as $ \alpha \rightarrow 0 $, we can assume that $ x_\alpha \rightarrow \tilde{x} $ as $ \alpha\rightarrow 0 $.
Hence, from the properties of the sup-convolution, the fact that $ u_i $ is USC and $ w_i $ is LSC, we can check that
$
\ell \le ( u_i (\tilde{x}) - w_i (\tilde{x}) )_+ \le || (u_i- w_i)_+||_{L^\infty (\partial Q_i)}.
$
The claim is proved and therefore, by the pointwise convergence of $ u_i^\alpha $ to $ u_i $, passing to the $\limsup $ as $ \alpha\rightarrow 0 $ in~\eqref{758}, we have
\[ 
( u_i (y) - w_i (y) )_+ \le || (u_i- w_i)_+||_{L^\infty (\partial Q_i)}, \quad \text{for all } y\in Q_{\alpha_0}.
\]
The above inequality holds for any $ y\in Q_{\alpha_0} $ with $ \alpha_0 $ arbitrarily chosen, then we obtain~\eqref{eq:local compar}. Finally, \eqref{eq:local compar1} is obtained by using the similarly arguments, the proof is complete.
\end{proof}

Thanks to the local comparison principle, apply the similar arguments in \cite[Theorem 3.3.4]{Oudet2016_these}, we now have the global comparison principle which leads us to the uniqueness of~\eqref{eq:HJmain}. 

\begin{thm}
	Under Assumptions $ [A] $ and $ [\tilde{A}3] $, let $ \cU=(u_1,\ldots,u_N,u_\Gamma)\ $ be a bounded,USC viscosity subsolution of~\eqref{eq:HJmain} and  $\cW= (w_1,\ldots,w_N, w_\Gamma) $ be a bounded,USC viscosity supersolution of~\eqref{eq:HJmain}. Then $ u_i \le w_i $ on $ \cP_i $ for all $ i\in \{1,\ldots,N\} $ and $ u_\Gamma \le w_\Gamma $  on $ \Gamma $.
\end{thm}

\begin{cor}
The value function $ \ccV $ satisfies $ (v_1,\ldots,v_N,\ccV|_\Gamma) $ is a unique viscosity solution of~\eqref{eq:HJmain} and $ v_i $ is continuous on $ \cP_i $ for all $ i\in \{1,\ldots,N\} $.
\end{cor}

%%%%%%%%%%%%%%%%%%%%%%%%%%%%%%%%%%%%%%%%%
%%%%%%%%%%%%%%%%%%%%%%%%%%%%%%%%%%%%%%%%%
\section{A more general optimal control problem}\label{sec7}
In this section, we generalize the control problem studied in Section \ref{A3-cond} by allowing some of the entry costs functions to be zero. The situation can be seen as intermediary between the one studied in~\cite{Oudet2016_these} when all the entry costs are zero, and that studied previously under the strong controllability condition $[A3]$, in this work when all the entry costs function are positive. This generalization holds under the moderate controllability condition 
$[\tilde{A}3]$, but we omit the details. 
Accordingly, every results presented below will mainly be obtained by combining the arguments proposed above with those used in~\cite{Oudet2016_these}. Hence, we will present the results without the proofs.

Specifically, we consider the optimal control problems with non-negative entry cost $ \zC = \{\zc_1, \ldots,\zc_m, \zc_{m+1}, \ldots \zc_N \}$, where $ \zc_i \equiv 0 $ if $ i\le m $ and $ \zc_i(x) >0 $ for all $ x\in \Gamma $ if $ i>m $, keeping all the assumptions and definitions of Section~\ref{sec2} unchanged. The value function associated to $ \zC $ is denoted by $ \zV $. Similarly to Lemma~\ref{lem:continuity}, $ \zV|_{\cP_i \backslash \Gamma}  $ is continuous and Lipschitz continuous in the neighborhood of $ \Gamma $. Hence, we can define $ \zv_i $ as follows 
\[
\zv_i(x) = 
\begin{cases}
\zV(x), & \text{if } x\in \cP_i \backslash \Gamma,\\
\ds \lim_{\cP_i \backslash \Gamma \ni z \rightarrow x} \zV(z), & \text{if } x\in \Gamma.
\end{cases}
\]
Additionally, it is clear that for all $ i,j\le m $, $ \zv_i(x) = \zv_j (x) $ for all $ x\in \Gamma $, i.e. $ \zV|_{\cup_{i\le m} \cP_i}  $ is a continuous function which will be noted $ \zV_c $ from now on.

Combining the arguments in \cite{Oudet2016_these} and in Section~\ref{sec3} leads us to the following lemma.
\begin{lem}
	The value function $ \zV $ satisfies
	\[
	 \max_{m<i\le N } \{\zv_i (x)\} \le \zV|_\Gamma (x) = \zV_c (x) \le  \min_{ m<i\le N } \{ \zv_i (x) + \zc_i (x)  \}, \quad  x\in \Gamma.
	\]

\end{lem}

We now define a set of admissible test-function  and the Hamilton-Jacobi
system that will characterize $\zV$.
\begin{defn}
	A function $\varphi:\left(\cup_{i=1}^{m}\cP_{i}\right) \times \cP_{m+1}\times\ldots\times \cP_{N}\rightarrow\mathbb{R}^{N-m+1}$ of the form  
	\[\varphi\left(x_{c},x_{m+1},\ldots,x_{N}\right)=\left(\varphi_{c}(x_{c}),\varphi_{m+1} (x_{m+1}),\ldots,\varphi_{N} (x_{N})\right)
	\]
	is an admissible test-function
	if 
	\begin{itemize}
		\item $\varphi_{c}$ is continuous and for $i\le m$, $\varphi_{c}|_{\cP_{i}}$
		belongs to $C^{1}\left(\cP_{i}\right)$,
		
		\item the space of admissible test-function is denoted by $R\left(\cS\right)$.
	\end{itemize}

We ready to define a viscosity sub and supersolution of the Hamilton-Jacobi system related to $ \zC $.

\end{defn}
\begin{defn}
	\label{def:vis-mix} A function $\zU=\left(\zu_{c},\zu_{m+1},\ldots,\zu_{N}\right)$,
	where $\zu_{c}\in USC\left(\cup_{j=1}^{m}\cP_{j};\mathbb{R}\right)$ and $\zu_{i}\in USC\left(\cP_{i};\mathbb{R}\right)$,
	is called a \emph{viscosity subsolution} of the Hamilton-Jacobi
	system if for any test-function $\varphi \in R\left(\cS\right)$:
	\begin{enumerate}
		\item 
		if  $\zu_{c}-\varphi_{c}$ has a local maximum  at $x_{c}\in \cup_{j=1}^{m}\cP_{j}$ and if
		\begin{itemize}
			\item    $x_{c}\in \cP_{j}\backslash\Gamma $ for some $j\le m$, then
			\[
			\begin{array}{c}
			\lambda \zu_{c}\left(x_{c}\right)+H_{j}\left(x,\partial \varphi_{c}\left(x_{c}\right)\right)\le0,\end{array}
			\]
			\item  $x_{c} \in \Gamma$, then
			\[
			\ds \lambda \zu_{c}\left(x_c\right)+\max\left\{ -\lambda \zu_c(x_c) ,\max_{j\le m}\left\{ H_{j}^{+}\left(x_c, \partial \varphi_{c} (x_c)\right)\right\}, H_\Gamma \left(x_c, \dfrac{\partial \varphi_c}{\partial e_0} (x_c) \right)  \right\} \le0;
			\]
		\end{itemize}
		\item  if $\zu_{i}-\varphi_{i}$ has a local maximum point  at $x_{i}\in \cP_{i}$ for $i>m$, and if
		\begin{itemize}
			\item  $x_{i}\in \cP_{i}\backslash \Gamma$, then 
			\[
			\begin{array}{c}
			\lambda \zu_{i}\left(x_{i}\right)+H_{i}\left(x_i,\partial\varphi_{i} (x_{i})\right)\le0,\end{array}
			\]
			\item  $x_{i} \in \Gamma$, then
			\[
			\ds\lambda \zu_i (x_i)+\max\left\{ H_{i}^{+}\left(x_i,\partial \varphi_i (x_i)\right), -\lambda \zu_c (x_i) \right\} \le0.
			\]
		\end{itemize}
	\end{enumerate}
	A function $\zU=\left(\zu_{c},\zu_{m+1},\ldots,\zu_{N}\right)$
	where $\zu_{c}\in LSC\left(\cup_{j=1}^{m}\cP_{j};\mathbb{R}\right),\zu_{i}\in LSC\left(\cP_{i};\mathbb{R}\right)$
	is called a \emph{viscosity supersolution} of the Hamilton-Jacobi
	system if for any test-function $\varphi\in R\left(\cS\right)$:
	\begin{enumerate}
		\item 
		if  $\zu_{c}-\varphi_{c}$ has a local minimum  at $x_{c}\in \cup_{j=1}^{m}\cP_{j}$ and if
		\begin{itemize}
			\item    $x_{c}\in \cP_{j}\backslash\Gamma $ for some $j\le m$, then
			\[
			\begin{array}{c}
			\lambda \zu_{c}\left(x_{c}\right)+H_{j}\left(x_c,\partial \varphi_{c}\left(x_{c}\right)\right)\ge0,\end{array}
			\]
			\item  $x_{c} \in \Gamma$, then
			\[
			\ds \lambda \zu_{c}\left(x_c\right)+\max\left\{ -\lambda \zu_c(x_c) ,\max_{j\le m}\left\{ H_{j}^{+}\left(x_c, \partial \varphi_{c} (x_c)\right)\right\}, H_\Gamma \left(x_c, \dfrac{\partial \varphi_c}{\partial e_0} (x_c) \right)  \right\} \ge 0;
			\]
		\end{itemize}
		\item  if $\zu_{i}-\varphi_{i}$ has a local minimum point  at $x_{i}\in \cP_{i}$ for $i>m$, and if
		\begin{itemize}
			\item  $x_{i}\in \cP_{i}\backslash \Gamma$, then 
			\[
			\begin{array}{c}
			\lambda \zu_{i}\left(x_{i}\right)+H_{i}\left(x_i,\partial\varphi_{i} (x_{i})\right)\ge0,\end{array}
			\]
			\item  $x_{i} \in \Gamma$, then
			\[
			\ds\lambda \zu_i (x_i)+\max\left\{ H_{i}^{+}\left(x_i,\partial \varphi_i (x_i)\right), -\lambda \zu_c (x_i) \right\} \ge0.
			\]
		\end{itemize}
	\end{enumerate}
	A function $\zU=\left(\zu_{c},\zu_{m+1},\ldots,\zu_{N}\right)$
	where $\zu_{c}\in C\left(\cup_{j\le m}\cP_{j};\mathbb{R}\right)$ and
	$\zu_{i}\in C\left(\cP_{i};\mathbb{R}\right)$ for all $i>m$ is called
	a \emph{viscosity solution} of the Hamilton-Jacobi system if
	it is both a viscosity subsolution and a viscosity supersolution
	of the Hamilton-Jacobi system.
\end{defn}

Similarly to the previous sections, we have the relation between the value function $ \zV $ and the Hamilton-Jacobi system.

\begin{thm}
	Let $ \zV $ be the value function corresponding to $ \zC $, then $  (\zv_c, \zv_{m+1}, \ldots \zv_N) $ is a viscosity solution of the Hamilton-Jacobi system.
\end{thm}

We now state and prove the comparison principle for the Hamilton-Jacobi system which is the main result of this section

\begin{thm}\label{thm4}
	Let $ \zU = (\zu_c, \zu_{m+1}, \ldots, \zu_N) $ and $ \zW = (\zw_c, \zw_{m+1}, \ldots, \zw_N)  $ be respectively a bounded viscosity sub and supersolution of the Hamilton-Jacobi system. Then $ \zU \le \zW $, i.e. $ \zu_c \le \zw_c $ on $ \cup_{j\le m} \cP_j $ and $ \zu_i \le \zw_i $ on $ \cP_i $ for all $ i>m $.
\end{thm}
\begin{proof}
First of all, similarly to the first step of the proof of Theorem~\ref{thm3} (Section~\ref{sec6}) we can check that  there exists $ (\phi_c, \phi_{m+1}, \ldots , \phi_N) $ such that for all $ 0<\mu<1 $ and $ \mu $ close to $ 1 $, $ \zu^\mu_c := \mu \zu +  (1-\mu) \phi_c$  and 	$ \zu^\mu_j := \mu \zu +  (1-\mu) \phi_j $ for all $ i\in \{m+1, \ldots,N\} $ then $ (\zu_c,\zu_{m+1},\ldots,\zu_N) $ is a viscosity subsolution of HJ system related to $ \zC $. We also have the function $ \zu^\mu_c  -\zw_c$ and $ \zu_i^\mu - \zw_i$ have maximum value $ M_c^\mu $ and $ M^\mu_j $ which are reached at some points $ \bar{x}_c $ and $ \bar{x}_j $. We argue by contradiction, through considering the two following cases:

\begin{enumerate}
	\item [Case A] Assume that $ M_i^\mu > M_c^\mu $ and $ M_i^\mu >0 $.  Repeat Case A as in the proof of Theorem~\ref{thm3}, implies that either $ \zu^\mu_i(\bar{x}_i) -\zw_i(\bar{x}_i) = M_i^\mu  \le 0 $ or $ M_i^\mu \le  M_c^\mu $, the desired contradiction.
	
	\item [Case B] Assume that $ M_c^\mu \ge \max_{j>m} \{M_j^\mu\} $ and $ M_c^\mu>0 $. If $ \zw_c (\bar{x}_c) < \min_{j>m} \{ \zw_j (\bar{x}_c) + \zc_j(\bar{x}_c) \}$, we can check that $ \zw_c $ satisfies the super-optimality  \cite[Theorem 3.2.4]{Oudet2016_these} at $ \bar{x}_c $. Furthermore, since $ \zu_c^\mu $ is a viscosity subsolution of
	\[ 
	\begin{cases}
	\lambda \zu_c^\mu (x) + H_i (x,\partial \zu_c^\mu (x)) \le 0, & \ds \text{if } x\in (\bigcup_{j\le m} \cP_j )\backslash \Gamma, \\
	\lambda \zu_c^\mu (x) + \max_{j\le m} H^+_j (x, \partial \zu_c^\mu (x)) \le 0, & \text{if } x\in  \Gamma,
	\end{cases}
	 \]
	 Thanks to the comparison principle for HJ when all of entry costs are zero,  see \cite[Theorem 3.2.5]{Oudet2016_these}, we obtain that $ \zu_c^\mu (\bar{x}_c)  - \zw_c (\bar{x}_c) = M_c^\mu \le 0$ which leads to a contradiction. Otherwise, if there exists $ k>m $ such that  
	 \begin{equation}\label{eq:gen}
	 \zw_c (\bar{x}_c) \ge \min_{j>m} \{ \zw_j (\bar{x}_c) + \zc_j(\bar{x}_c) \}  =   \zw_k (\bar{x}_c) + \zc_k(\bar{x}_c), 
	 \end{equation}
	 repeat Case B.3 as in in the proof of Theorem~\ref{thm3}, implies that there exists $ k> m  $ such that
	 \[ 
	  M^\mu_k = \zu^\mu_k (\bar{x}_c) -\zw_k (\bar{x}_c) = \zu^\mu_c (\bar{x}_c) -\zw_c (\bar{x}_c) = M^\mu_c.
	  \]
	  Repeat Case A and note that thanks to~\eqref{eq:gen}, the case $ w_k (\bar{x}_c) \ge \zw_c (\bar{x}_c)$ does not occur, we finally obtain that $ M_k^\mu \le 0 $, the desired contradiction. 
 \end{enumerate}
	Finally, we get $ M_c^\mu \le 0 $ and $ M_i^\mu \le 0 $ for all $ i\in \{1,\ldots,N\} $ and $ 0< \mu <1 $, $ \mu $ close to $ 1 $. We conclude by letting $ \mu  $ to $ 1 $.
\end{proof}

%%%%%%%%%%%%%%%%%%%%%%%%%%%%%%%%%%%%%%%%%
%%%%%%%%%%%%%%%%%%%%%%%%%%%%%%%%%%%%%%%%%

\section{Appendix}

\subsection{Proof of Lemma~\ref{lem:fl=FL}}
By the piecewise linearity of the function $(\zeta,\xi) \mapsto -(\zeta \cdot e_0) \frac{\partial \varphi}{\partial e_0} (x) - \xi $, we can obviously see that Lemma~\ref{lem:fl=FL} is a directly consequence of the following lemma. 
\begin{lem}
Under Assumption $[A]$, for $x\in\Gamma$, then 
\[
\textsc{FL}_\Gamma(x) \subset f\ell_\Gamma (x) \subset \overline{\text{co}} \left( \textsc{FL}_\Gamma (x) \right),
\]
where $\textsc{FL}_\Gamma(x)$ is defined in \eqref{eq:FL}.
\end{lem}
\begin{proof}
In this case the proof is standard (see \cite{BD1997}), but we prove it for completeness.

The inclusion $\textsc{FL}_\Gamma(x) \subset f\ell_\Gamma (x) $ is obtained by explicitly constructing trajectories which still remain on the interface $\Gamma$ for a short time. More specifically, let $i\in\{1,\ldots,N\}$ and $a\in A_i$ such that $(f_i(x,a),\ell_i(x,a)) \in \textsc{FL}_i \cap (\R e_0 \times \R)$. Then from the strong controllability condition $[A3]$ or the moderate controllability $ [\tilde{A3}] $, there exists $ (y_\alpha, \alpha)\in \cT_x $ such that $y_{x}$ remains on $\Gamma$ at least for a while and satisfies
\[
\dot{y}_{x}(0) = f(y_{x,\alpha} (0), \alpha (0)) = f_i (x,a).
\]
Thus, there exists a sequence of times $t_n\rightarrow 0^+$ such that $y_{x} (t) \in \Gamma $ for all  $t\in [0,t_n]$  and
\[
f_i(x,a) = \lim_{n\rightarrow0}\dfrac{1}{t_n} \int_0^{t_n} f(y_{x}(t),\alpha(t)) dt. 
\]
Similarly, we get
\[
\ell_i(x,a) = \lim_{n\rightarrow0}\dfrac{1}{t_n} \int_0^{t_n} \ell(y_{x}(t),\alpha(t)) dt. 
\]
Hence $(f_i(x,a), \ell_i(x,a)) \in  f\ell_\Gamma (x)$ and the first inclusion is proved.

We now consider $(\zeta,\xi)\in  f\ell_\Gamma (x)$, there exists a sequence of admissible trajectories $(y_n,\alpha_n)\in \cT_x$ which remain on $\Gamma$ at least for a while and a sequence of times $t_n\rightarrow 0^+$ such that
\[
 \lim_{n\rightarrow0}\dfrac{1}{t_n} \int_0^{t_n} f(y_{n}(t),\alpha_n(t)) dt = \zeta,\quad \lim_{n\rightarrow0}\dfrac{1}{t_n} \int_0^{t_n} \ell(y_{n}(t),\alpha_n(t)) dt=\xi.
\]
On the other hand

\begin{align*}
 &(\zeta_n, \xi_n) :=\dfrac{1}{t_n} \left( \int_0^{t_n}f(y_{n}(t),\alpha_n(t)) dt , \int_0^{t_n} \ell(y_{n}(t),\alpha_n(t)) dt\right) \\
 & =   \dfrac{1}{\left| \{ t\le t_n : y_n(t) \in \Gamma \}\right| } \left( \int_0^{t_n}f(x,\alpha_n(t)) \mathds{1}_{\{y_n(t)\in \Gamma\}} dt , \int_0^{t_n} \ell(x,\alpha_n(t))\mathds{1}_{\{y_n(t)\in \Gamma\}} dt\right) + o(1),
\end{align*}
where $o(1)$ is a vector tending to $0$ as $n\rightarrow \infty$.
Therefore, the distance of 
\[
 \frac{1}{\left| \{ t\le t_n : y_n(t) \in \Gamma \}\right| } \left( \int_0^{t_n}f(x,\alpha_n(t)) \mathds{1}_{\{y_n(t)\in \Gamma\}} dt , \int_0^{t_n} \ell(x,\alpha_n(t))\mathds{1}_{\{y_n(t)\in \Gamma\}} dt\right)
 \]
to the set $\overline{co} \left( \cup_{i=1}^N \textsc{FL}_i (x) \right)$ tends to $0$ as $n\rightarrow \infty$.
Additionally, since $f(y_n(t),\alpha_n (t))\in \Gamma$ for a.e. $t\in [0,t_n]$, we get the distance of $(\zeta_n,\xi_n)$ to the set $\overline{co}\left( \textsc{FL}_\Gamma (x) \right)$ tends to $0$ as $n\rightarrow \infty$.
Hence, we obtain the second inclusion $ f\ell_\Gamma (x) \subset \overline{\text{co}} \left( \textsc{FL}_\Gamma (x) \right)$.
\end{proof}

\subsection{A second proof of Theorem~\ref{thm3}}
First of all, from Definition~\ref{def:vis-sol}, we see that if $U = \{u_1 , \ldots, u_N , u_\Gamma\}$ is a viscosity subsolution of~\eqref{eq:HJmain}, then for any $i\in\{1,\ldots,N\}$, $u_i$ is a viscosity subsolution of~\eqref{001}
and $u_i(x) \le u_\Gamma (x)$  for any $x\in \Gamma$.
Therefore, we have the following properties for $U$ (see \cite{Oudet2016_these} for detailed proofs)

\begin{lem}\label{lem:sub-optimal}
Under Assumptions $[A]$ and $[A3]$, let $i\in \{1,\ldots,N\}$ and $x\in \left( B(\Gamma,r) \cap \cP_i \right) \backslash \Gamma$ and $\alpha_i \in L^{\infty} (0,\infty; A_i)$.
Let $\eta>0$ be such that $y_x(t)= x + \int_0^t f_i (y_x(s),\alpha_i(s) ds $ belongs to $ B(\Gamma,r) \cap \cP_i$ for any $t\in [0,\eta]$.
Assume that $u_i$ is a $C^1$ function on $\Gamma$. Then we have
\begin{equation}
u_i(x) \le \int_0^\eta \ell	_i (y_x(s),\alpha_i(s) e^{-\lambda s} ds + u_i (y_x(\eta)) e^{-\lambda \eta}.
\end{equation}
\end{lem}

\begin{lem} \label{lem:regu}
Assume $[A]$ and $[A3]$ and that $ u_i \in C (\cP_i) $ is bounded for all $ i\in \{1,\ldots,N\} $. Let $ \{\rho_\varepsilon\}_{\varepsilon} $ be a sequence of mollifiers defined on $ \R $ as follows
\[ 
\rho_\varepsilon (x) = \dfrac{1}{\varepsilon} \rho (\dfrac{x}{\varepsilon}),
\]
where
\[ 
\rho \in C^{\infty} (\R,\R^+),\quad \int_\R \rho(x) dx =1 \quad \text{and } \supp(\rho) \subset [-1,1].   
\]
We consider the function $ u_i^\varepsilon $ defined on $ \cP_i $ by
\[
u_i^\varepsilon  = u_i \star \rho_\varepsilon (x) = \int_\R u_i(x^i e_i + (x^0 - t)e_0 ) p_\varepsilon (t) dt, \quad \text{where } x=(x^i,x^0). 
\]
Then $ u_\varepsilon $ converges uniformly to $ u $  in $ L^\infty (\cP_i, \R) $ and there exists a function $ m:(0,+\infty) \rightarrow (0,+\infty) $ such that $ \lim_{\varepsilon\rightarrow 0 } m(\varepsilon) = 0 $ and the function $ u_\varepsilon - m(\varepsilon) $ is a viscosity subsolution of~\eqref{001} on a neighborhood of $ \Gamma $.
\end{lem}

\begin{rem}
For a proof of Lemma~\ref{lem:sub-optimal} and Lemma~\ref{lem:regu} see  \cite[Section 3.2.3 and Section 3.2.4]{Oudet2016_these}. 
The proof in Achdou, Oudet \& Tchou \cite{AOT2015} is adopted to prove Lemma~\ref{lem:sub-optimal} and finally, the existence of a function $ m: (0,+\infty) \rightarrow (0,+\infty) $ with $ m(0^+) =0 $ such that $ u_i^\varepsilon - m (\varepsilon) $ is a subsolution of~\eqref{001} on a neighborhood of $ \Gamma $ was proved by Lions~\cite{Lions1982} or Barles \& Jakobsen~\cite{BJ2002}.    
	
\end{rem}
Next, we have the important property of viscosity supersolution.
\begin{lem}\label{lem:super-optimal}
	Let assumptions $[A]$ and $[A3]$ hold  and let $i\in \{1,\ldots,N\}$ and $\bar{x}\in \Gamma$.
	Let $ \varphi \in C^1(\cP_i) $ be such that $ w_i-\varphi $ has a local maximum point at $ x $.
	If the function $ w_i $ satisfies
	\begin{equation}\label{501} 
	w_i(\bar{x}) < w_\Gamma(\bar{x}),
	\end{equation}
	then either (\textbf{H1}) or (\textbf{H2}) below is true:
	\begin{description}
		\item[(H1)]There exists a sequence $ \{\eta_k\} $ of positive real numbers such that $ \lim_{k\rightarrow + \infty} \eta_k = \eta >0$ and a sequence $ \{x_k\} \subset \cP_i \backslash \Gamma $ such that $ \lim_{k\rightarrow \infty} x_k = \bar{x} $ and for each $ k $, there exists a control law $ \alpha_i^k $ such that the corresponding trajectory $ y_{x_k} (s) \in \cP_i $ for all $ s\in [0,\eta_k] $ and
		\[ 
		w_i(x_k) \ge \int_0^{\eta_k} \ell_i \left( y_{x_k}(s), \alpha_i^k (s)\right) e^{-\lambda s} ds 
		+ w_i (y_{x_k}  (\eta_k)) e^{-\lambda \eta_k}. 
		\] 
		
		\item [(H2)] 
		\[ 
		\lambda w_i(\bar{x}) + \max_{a \in A^\Gamma_i} \{- f_i(\bar{x},a) \cdot \partial \varphi (\bar{x}) - \ell_{i} (\bar{x},a) \} \ge 0.
		\]
	\end{description}
	
\end{lem}

\begin{proof}
	Since $ w_i $ and $ w_\Gamma $ are continuous on $ \Gamma $, from \eqref{501}, there exists $ \varepsilon>0 $ such that  $ \lambda w_i (x) < u_\Gamma (x)  $ for all $ x\in \Gamma  $ and $ | x - \bar{x} | < \varepsilon $. From the definition of supersolution, we get
	\[
	\lambda w_i (x) + H^+_i (x, \partial w_i (x)) \ge 0, \quad \text{for all }   x\in \Gamma \cap B(\bar{x}, \varepsilon) ,
	\]
	in the viscosity sense. Hence, from \cite[Theorem 3.2.4]{Oudet2016_these}, the proof is complete.
\end{proof}
We are ready to give the second proof of Theorem~\ref{thm3}.

\begin{proof}[A second proof of Theorem~\ref{thm3}]
	Let $ u_i^\varepsilon $ be the approximation of $ u_i $ given by Lemma~\ref{lem:regu}.
	Similarly to the first step of the proof of Theorem~\ref{thm3} (Section~\ref{sec6}), we set $ u^{\varepsilon,\mu}_{i} := \mu (u^\varepsilon_i - m(\varepsilon)) + (1-\mu) \psi_i $ for all $ i\in \{1,\ldots,N\} $ and $ u^\mu_\Gamma := \mu u_\Gamma + (1-\mu)\psi_\Gamma $, $ \mu \rightarrow 1^- $ then $ (u^{\varepsilon,\mu}_{1}, \ldots, u^{\varepsilon,\mu}_{N}, u^\mu_\Gamma) $ is a viscosity subsolution of~\eqref{eq:HJmain}. 
	Let $ M^{\varepsilon,\mu}_i $ and $ M^{\mu}_\Gamma $ be respectively the maximal value of $ u^{\varepsilon,\mu}_{i} - w_i$ and $ u^\mu_\Gamma - w_\Gamma $ which are reached at some point $ \bar{x}^{\varepsilon,\mu}_i \in \cP_i$ and $ \bar{x}^\mu_\Gamma \in \Gamma $. 
	We argue by contradiction, through considering the two following cases:
	
	\begin{enumerate}
		\item [Case A] Assume that $ M_i^{\varepsilon,\mu} > M_\Gamma^\mu$ and $ M_i^{\varepsilon,\mu} > 0 $.  If $  \bar{x}^{\varepsilon,\mu}_i $, the proof is done by using the classical techniques. Therefore, we only focus on the case when $ \bar{x}^{\varepsilon,\mu}_i  \in \Gamma$.
		From Lemma~\ref{lem:continuity-sub}, $ u^{\varepsilon,\mu}_{i} $ is Lipschitz in a neighborhood of $ \Gamma $.
		Moreover, from Lemma~\ref{lem:regu}, $ u^{\varepsilon,\mu}_{i}|_\Gamma \in C^1(\Gamma) $.
		Then, we can construct two test-function $\psi, \varphi\in C^1(\cP_i) $ such that 
		\[
		\psi (x) =  u^{\varepsilon,\mu}_{i} (0,x^0) + Cx^i, \quad \varphi (x) = u^{\varepsilon,\mu}_{i} (0,x^0) - Cx^i,
		\]
		where $ x=(x^i,x^0) $ and $ C $ is a positive constant which is large enough. 
		This implies that $\psi|_\Gamma =  \varphi|_\Gamma = u^{\varepsilon,\mu}_{i}|_\Gamma $,  $ \psi $ remains above and $ \varphi $ remains below $ u^{\varepsilon,\mu}_{i} $ on a neighborhood of $ \Gamma $.
		Since $ M^{\varepsilon,\mu}_i>0 $, it is easy to see that $ w_i - \varphi $ has a local minimum at $ \bar{x}^{\varepsilon,\mu}_i $.
		Next, we claim that  $ w_i( \bar{x}^{\varepsilon,\mu}_i ) < w_\Gamma( \bar{x}^{\varepsilon,\mu}_i )$. Indeed, if it does not hold, then we have
		\[ 
		w_\Gamma ( \bar{x}^{\varepsilon,\mu}_i) \le w_i ( \bar{x}^{\varepsilon,\mu}_i) < u_i^{\varepsilon,\mu} ( \bar{x}^{\varepsilon,\mu}_i) \le u_\Gamma^\mu ( \bar{x}^{\varepsilon,\mu}_i),
		\]
		and this implies $ M_i^{\varepsilon,\mu} \le M_\Gamma^\mu $, the desired contradiction.
		Hence, from Lemma~\ref{lem:super-optimal}, we have the two following cases:
		\begin{description}
			\item[(H1)] $ \lambda w_i (\bar{x}^{\varepsilon,\mu}_i) + \max_{a \in A^\Gamma_i} \{- f_i(\bar{x}^{\varepsilon,\mu}_i,a) \cdot \partial \varphi (\bar{x}^{\varepsilon,\mu}_i) - \ell_{i} (\bar{x}^{\varepsilon,\mu}_i,a) \} \ge 0$.
			Moreover, since $ u^{\varepsilon,\mu}_{i} $ is a subsolution of~\eqref{eq:HJG5}, $   u^{\varepsilon,\mu}_{i}- \psi$ has a local maximum point at $ \bar{x}^{\varepsilon,\mu}_i $, $\psi|_\Gamma =   u^{\varepsilon,\mu}_{i}|_\Gamma $ and $ H_i(\bar{x}^{\varepsilon,\mu}_i, p)  \ge \max_{a\in A_i^\Gamma}\{-f_i (\bar{x}^{\varepsilon,\mu}_i, a) \cdot p - \ell_i (\bar{x}^{\varepsilon,\mu}_i, a) \}$, we get
			\[
			\lambda u^{\varepsilon,\mu}_{i} (\bar{x}^{\varepsilon,\mu}_i) + \max_{a \in A^\Gamma_i} \{- f_i(\bar{x}^{\varepsilon,\mu}_i,a) \cdot \partial \varphi (\bar{x}^{\varepsilon,\mu}_i) - \ell_{i} (\bar{x}^{\varepsilon,\mu}_i,a) \} \le 0.
			\]
			Hence, $ u^{\varepsilon,\mu}_{i} (\bar{x}^{\varepsilon,\mu}_i)  \le w_i  (\bar{x}^{\varepsilon,\mu}_i)$ and thus $ M^{\varepsilon,\mu}_i \le 0 $, which leads to a contradiction.
			
			\item[(H2)] With the notation in Lemma~\ref{lem:super-optimal}, we have
			\[
			w_i(x_k) \ge \int_0^{\eta_k} \ell_i \left( y_{x_k}(s), \alpha_i^k (s)\right) e^{-\lambda s} ds 
			+ w_i (y_{x_k}  (\eta_k)) e^{-\lambda \eta_k}.  
			\]
			On the other hand, from Lemma~\ref{lem:sub-optimal}
			\[
			u^{\varepsilon,\mu}_{i} (x_k) \le \int_0^{\eta_k} \ell_i \left( y_{x_k}(s), \alpha_i^k (s)\right) e^{-\lambda s} ds 
			+ u^{\varepsilon,\mu}_{i} (y_{x_k}  (\eta_k)) e^{-\lambda \eta_k}.  
			\]
			This implies
			\[
			u^{\varepsilon,\mu}_{i} (x_k) - w_i (x_k) \le ( u^{\varepsilon,\mu}_{i}(\eta_k) - w_i (\eta_k) ) e^{-\lambda \eta_k}.
			\]
			Letting $ k\rightarrow \infty $, one gets $ M^{\varepsilon,\mu}_i \le M^{\varepsilon,\mu}_i e^{-\lambda \eta} $ and thus $ M^{\varepsilon,\mu}_i<0 $, which is a contradiction.
		\end{description}
		
		\item [Case B] Assume that $ M_\Gamma^{\mu} \ge \max_j \{M^{\varepsilon,\mu}_j\} $ and $  M_\Gamma^{\mu}>0 $. If the case $ w_\Gamma (\bar{x}_\Gamma^\mu) < \min_j \{w_j(\bar{x}_\Gamma^\mu) + c_j(\bar{x}_\Gamma^\mu)\} $ occurs, we can apply the similar arguments as in Case B.1 of the first proof of Theorem~\ref{thm3} to obtain $ M_\Gamma^\mu \le 0 $, the desired contradiction. Otherwise, if the following case holds:
		\begin{equation} \label{018}
		w_\Gamma (\bar{x}_\Gamma^\mu) \ge  \min_j \{w_j(\bar{x}_\Gamma^\mu) + c_j(\bar{x}_\Gamma^\mu)\} = w_k(\bar{x}_\Gamma^\mu) + c_k(\bar{x}_\Gamma^\mu), 
		\end{equation}
		for some $ k\in \{1,\ldots,N\} $ and by the fact that 
		\[ 
		\min_j \{u_j^{\varepsilon,\mu}(\bar{x}_\Gamma^\mu) + c_j(\bar{x}_\Gamma^\mu)\} \ge u_\Gamma^\mu (\bar{x}_\Gamma^\mu) > w_\Gamma (\bar{x}_\Gamma^\mu),
		\] we obtain $ M_\Gamma^\mu = M_k^{\varepsilon,\mu} = w_k(\bar{x}_\Gamma^\mu) - u_k^{\varepsilon,\mu}(\bar{x}_\Gamma^\mu)$.
		Now we repeat Case A with the index $ i $ replacing $ k $, $ \bar{x}_i^{\varepsilon,\mu} $ replacing  $ \bar{x}_\Gamma^{\mu} $ and note that  $ w_k( \bar{x}_\Gamma^{\mu}) < w_\Gamma ( \bar{x}_\Gamma^{\mu}) $ because of~\eqref{018}. This implies $ M_k^{\varepsilon,\mu} = M_\Gamma^\mu \le 0$, which leads to a contradiction.
	\end{enumerate}
	
Finally, we conclude by letting $ \varepsilon\rightarrow 0  $ and $ \mu \rightarrow 1 $.
 \end{proof}

%	\bibliography{khang}

\begin{thebibliography}{10}
	
	\bibitem{ACCT2011}
	Y.~Achdou, F.~Camilli, A.~Cutr\`\i, and N.~Tchou.
	\newblock Hamilton-{J}acobi equations on networks.
	\newblock {\em IFAC proceedings volumes}, 44(1):2577--2582, 2011.
	
	\bibitem{ACCT2013}
	Y.~Achdou, F.~Camilli, A.~Cutr\`\i, and N.~Tchou.
	\newblock Hamilton-{J}acobi equations constrained on networks.
	\newblock {\em NoDEA Nonlinear Differential Equations Appl.}, 20(3):413--445,
	2013.
	
	\bibitem{AOT2015}
	Y.~Achdou, S.~Oudet, and N.~Tchou.
	\newblock Hamilton-{J}acobi equations for optimal control on junctions and
	networks.
	\newblock {\em ESAIM Control Optim. Calc. Var.}, 21(3):876--899, 2015.
	
	\bibitem{BD1997}
	M.~Bardi and I.~Capuzzo-Dolcetta.
	\newblock {\em Optimal control and viscosity solutions of
		{H}amilton-{J}acobi-{B}ellman equations}.
	\newblock Systems \& Control: Foundations \& Applications. Birkh\"auser Boston,
	Inc., Boston, MA, 1997.
	
	\bibitem{Barles1993}
	G.~Barles.
	\newblock Discontinuous viscosity solutions of first-order {H}amilton-{J}acobi
	equations: a guided visit.
	\newblock {\em Nonlinear Anal.}, 20(9):1123--1134, 1993.
	
	\bibitem{Barles1994}
	G.~Barles.
	\newblock {\em Solutions de viscosit\'e des \'equations de
		{H}amilton-{J}acobi}, volume~17 of {\em Math\'ematiques \& Applications
		(Berlin) [Mathematics \& Applications]}.
	\newblock Springer-Verlag, Paris, 1994.
	
	\bibitem{BBC2013}
	G.~Barles, A.~Briani, and E.~Chasseigne.
	\newblock A {B}ellman approach for two-domains optimal control problems in
	{$\mathbb R^N$}.
	\newblock {\em ESAIM Control Optim. Calc. Var.}, 19(3):710--739, 2013.
	
	\bibitem{BBC2014}
	G.~Barles, A.~Briani, and E.~Chasseigne.
	\newblock A {B}ellman approach for regional optimal control problems in
	{$\mathbb{R}^N$}.
	\newblock {\em SIAM J. Control Optim.}, 52(3):1712--1744, 2014.
	
	\bibitem{BB2018}
	G.~Barles, and E.~Chasseigne.
	\newblock An illustrated guide of the modern approaches of 
	{H}amilton-{J}acobi equations and control problems with discontinuities.
	\newblock { Preprint}. \url{https://arxiv.org/abs/1812.09197}
	
	\bibitem{BJ2002}
	G.~Barles and E.~R. Jakobsen.
	\newblock On the convergence rate of approximation schemes for
	{H}amilton-{J}acobi-{B}ellman equations.
	\newblock {\em M2AN Math. Model. Numer. Anal.}, 36(1):33--54, 2002.
	
	\bibitem{BH2007}
	A.~Bressan and Y.~Hong.
	\newblock Optimal control problems on stratified domains.
	\newblock {\em Netw. Heterog. Media}, 2(2):313--331, 2007.
	
	\bibitem{CM2013}
	F.~Camilli and C.~Marchi.
	\newblock A comparison among various notions of viscosity solution for
	{H}amilton-{J}acobi equations on networks.
	\newblock {\em J. Math. Anal. Appl.}, 407(1):112--118, 2013.
	
	\bibitem{CSM2013}
	F.~Camilli, D.~Schieborn, and C.~Marchi.
	\newblock Eikonal equations on ramified spaces.
	\newblock {\em Interfaces Free Bound.}, 15(1):121--140, 2013.
	
	\bibitem{Dao2019}
	M.-K. Dao.
	\newblock Hamilton-{J}acobi equations for optimal control on networks with
	entry or exit costs.
	\newblock {\em ESAIM Control Optim. Calc. Var.}, 25:Art. 15, 31, 2019.
	
	\bibitem{FM2013a}
	H.~Frankowska and M.~Mazzola.
	\newblock Discontinuous solutions of {H}amilton-{J}acobi-{B}ellman equation
	under state constraints.
	\newblock {\em Calc. Var. Partial Differential Equations}, 46(3-4):725--747,
	2013.
	
	\bibitem{GHZ2017}
	P.~J. Graber, C.~Hermosilla, and H.~Zidani.
	\newblock Discontinuous solutions of {H}amilton-{J}acobi equations on networks.
	\newblock {\em J. Differential Equations}, 263(12):8418--8466, 2017.
	
	\bibitem{HZ2015}
	C.~Hermosilla and H.~Zidani.
	\newblock Infinite horizon problems on stratifiable state-constraints sets.
	\newblock {\em J. Differential Equations}, 258(4):1430--1460, 2015.
	
	\bibitem{IMZ2013}
	C.~Imbert, R.~Monneau, and H.~Zidani.
	\newblock A {H}amilton-{J}acobi approach to junction problems and application
	to traffic flows.
	\newblock {\em ESAIM Control Optim. Calc. Var.}, 19(1):129--166, 2013.
	
	\bibitem{Ishii2013}
	H.~Ishii.
	\newblock A short introduction to viscosity solutions and the large time
	behavior of solutions of {H}amilton-{J}acobi equations.
	\newblock In {\em Hamilton-{J}acobi equations: approximations, numerical
		analysis and applications}, volume 2074 of {\em Lecture Notes in Math.},
	pages 111--249. Springer, Heidelberg, 2013.
	
	\bibitem{Lions1982}
	P.-L. Lions.
	\newblock {\em Generalized solutions of {H}amilton-{J}acobi equations},
	volume~69 of {\em Research Notes in Mathematics}.
	\newblock Pitman (Advanced Publishing Program), Boston, Mass.-London, 1982.
	
	\bibitem{LS2016}
	P.-L. Lions and P.~Souganidis.
	\newblock Viscosity solutions for junctions: well posedness and stability.
	\newblock {\em Atti Accad. Naz. Lincei Rend. Lincei Mat. Appl.},
	27(4):535--545, 2016.
	
	\bibitem{LS2017}
	P.-L. Lions and P.~Souganidis.
	\newblock Well-posedness for multi-dimensional junction problems with
	{K}irchoff-type conditions.
	\newblock {\em Atti Accad. Naz. Lincei Rend. Lincei Mat. Appl.},
	28(4):807--816, 2017.
	
	\bibitem{MW1967}
	E.~J. McShane and R.~B. Warfield, Jr.
	\newblock On {F}ilippov's implicit functions lemma.
	\newblock {\em Proc. Amer. Math. Soc.}, 18:41--47, 1967.
	
	\bibitem{Nakayasu2014}
	A.~Nakayasu.
	\newblock Metric viscosity solutions for {H}amilton-{J}acobi equations of
	evolution type.
	\newblock {\em Adv. Math. Sci. Appl.}, 24(2):333--351, 2014.
	
	\bibitem{Oudet2016_these}
	S.~Oudet.
	\newblock {\em Hamilton-Jacobi equations on networks or heterogeneous
		structures}.
	\newblock PhD thesis, Universit\'e de Rennes 1, 2016. \url{http://theses.fr/2015REN1S051}
	
	\bibitem{SC2013}
	D.~Schieborn and F.~Camilli.
	\newblock Viscosity solutions of {E}ikonal equations on topological networks.
	\newblock {\em Calc. Var. Partial Differential Equations}, 46(3-4):671--686,
	2013.
	
\end{thebibliography}
%\bibliographystyle{abbrv} 

\end{document}